\documentclass[12pt]{article}
 
\usepackage[margin=1in]{geometry} 
\usepackage{amsmath,amsthm,amssymb,scrextend}
\usepackage{fancyhdr}
\usepackage{graphicx}
\usepackage{latexsym}
\usepackage{amsfonts, manfnt}
\usepackage{setspace}
\usepackage{eucal}
\usepackage[all, poly, knot]{xy}
\usepackage{verse, play}
\input xy
\xyoption{all}
\xyoption{knot}
\xyoption{arc}

\newtheorem{theorem}{Theorem}

\theoremstyle{definition}
\newtheorem{definition}[theorem]{Definition}

\theoremstyle{remark}

\numberwithin{equation}{section}
\theoremstyle{plain}

\newtheorem{corollary}[theorem]{Corollary}

\newtheorem{proposition}[theorem]{Proposition}

 \author{Jason Cantarella, Allison Henrich, Elsa Magness, Oliver O'Keefe, Kayla Perez,\\
  Eric J. Rawdon and Briana Zimmer}
 \title{Knot Fertility and Lineage}
\begin{document}
 

\lhead{Knot Fertility and Lineage}
 
\maketitle

\begin{abstract} In this paper, we introduce a new type of relation between knots called the descendant relation. One knot $H$ is a {\em descendant} of another knot $K$ if $H$ can be obtained from a minimal crossing diagram of $K$ by some number of crossing changes. We explore properties of the descendant relation and study how certain knots are related, paying particular attention to those knots, called {\em fertile knots}, that have a large number of descendants. Furthermore, we provide computational data related to various notions of knot fertility and propose several open questions for future exploration.
\end{abstract}

\section{Introduction}

The $7_6$ knot, pictured in Figure~\ref{7-6}, is a particularly interesting knot. This is because, in a certain sense, all smaller knots are contained in this knot. 

\begin{figure}[!htbp] 
\begin{center}
\includegraphics[height=1.3in]{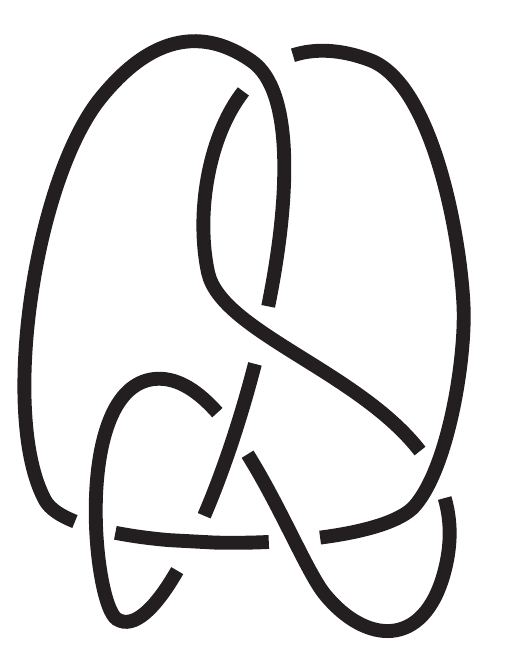} 
\end{center}
\vspace{-.3in}
\caption{A minimal crossing diagram of the $7_6$ knot.}\label{7-6}
\end{figure}

What do we mean by ``contained" in this context? We are interested in studying when a knot is a {\em parent} of another knot, where parenthood is defined as follows.

\begin{definition} A knot $K$ is a \textbf{parent} of a knot $H$ if a subset of the crossings in a minimal crossing diagram of $K$ can be changed to produce a diagram of $H$. In this case, we say that $H$ is a \textbf{descendant} of $K$. 
\end{definition}

For instance, knot $11a135$ is a parent of knot $3_1$, the trefoil, as shown in Figure~\ref{parent_ex}. Equivalently, we can say that the trefoil is a descendant of $11a135$. A curious feature of this definition is that any knot $K$ is both its own descendant and its own parent.

\begin{figure}[!htbp] 
\begin{center}
\includegraphics[height=1.7in]{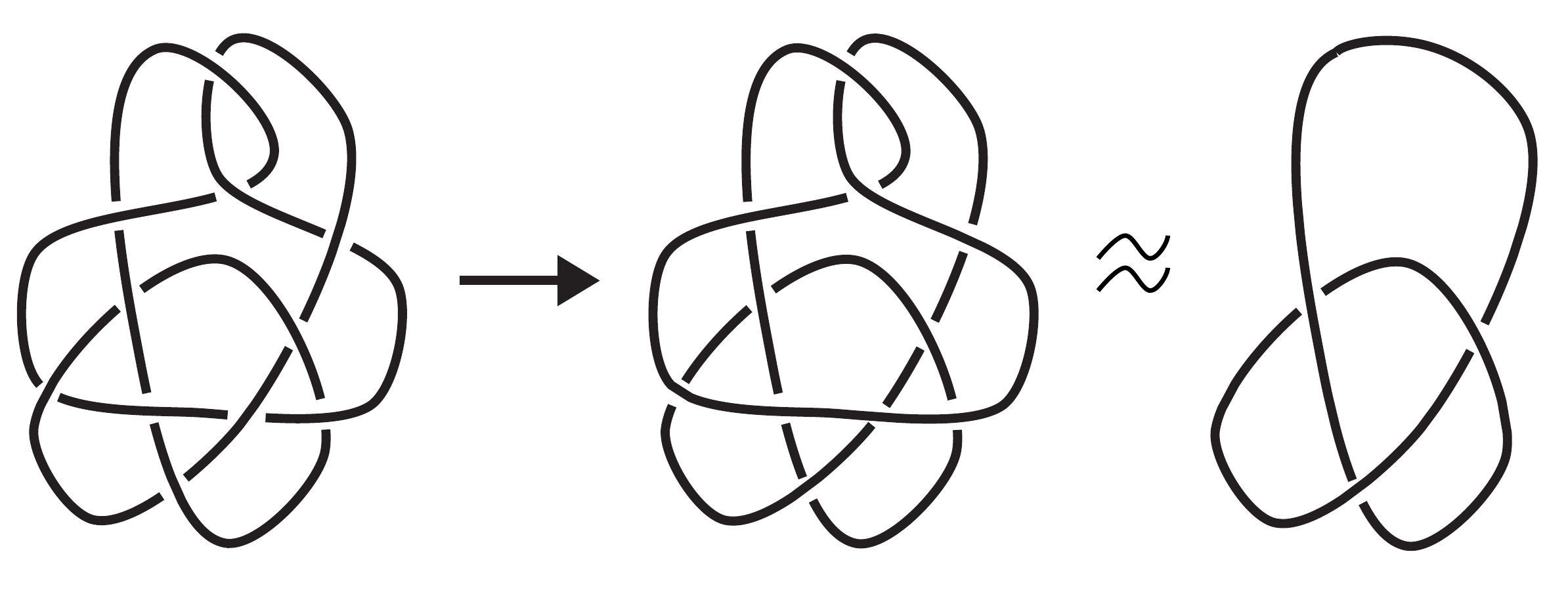} 
\end{center}
\vspace{-.3in}
\caption{A minimal crossing diagram of the $11a135$ knot becomes the trefoil after several crossing changes.}\label{parent_ex}
\end{figure} 

It is when we consider this more general relationship between knots that we see how interesting $7_6$ is, for $7_6$ is a parent of all knots with strictly smaller crossing number. That is, the descendants of $7_6$ are: $0_1$, $3_1$, $4_1$, $5_1$, $5_2$, $6_1$, $6_2$, and $6_3$. It is precisely for this reason that we call $7_6$ a {\em fertile} knot.

\begin{definition} A knot $K$ with crossing number $n$ is \textbf{fertile} if $K$ is a parent of every knot with crossing number less than $n$.
\end{definition}

Now that we have defined the concept of fertility, it is natural to ask, ``Which other knots are fertile?'' and ``Are there any fertile knots with more than seven crossings?'' In thinking about the first question, we note that it is a straightforward exercise to verify that $0_1$, $3_1$, $4_1$, $5_2$, $6_2$, and $6_3$ are also fertile. In Section~\ref{families}, we will develop tools to prove that knots such as $5_1$, $6_1$, $7_1$, and $7_2$ are {\em not} fertile. The computational results we present in Section~\ref{computations} illustrate that many other knots with seven or more crossings fail to be fertile. Answering the second question is somewhat trickier. We will introduce the related notions of $n$-fertility and $(n,m)$-fertility in Section~\ref{n-fertile} to reframe this question.

In addition to considering questions that specifically pertain to fertility, we consider more broadly the parent--descendant relationships between knots. For instance, is this relationship {\em transitive}? That is, if $A$ is a descendant of $B$ and $B$ is a descendant of $C$, is $A$ a descendant of $C$? We return to this fascinating question in Section~\ref{computations}.

\section{Preliminaries}

Before we get started, we note that it is often convenient to work with an alternative, but equivalent definition of what it means for one knot to be a parent (or descendant) of another. Suppose $K$ is a parent of $H$. Then if we consider all minimal diagrams of $K$ and {\em forget} these diagrams' crossing information---so we merely consider their {\bf shadows}---then $H$ must be obtainable by some choice of crossing information in one of these shadows. Figure~\ref{shadow} gives an example of how we derive a descendant from a parent following this process.

\begin{figure}[!htbp] 
\begin{center}
\includegraphics[height=1.3in]{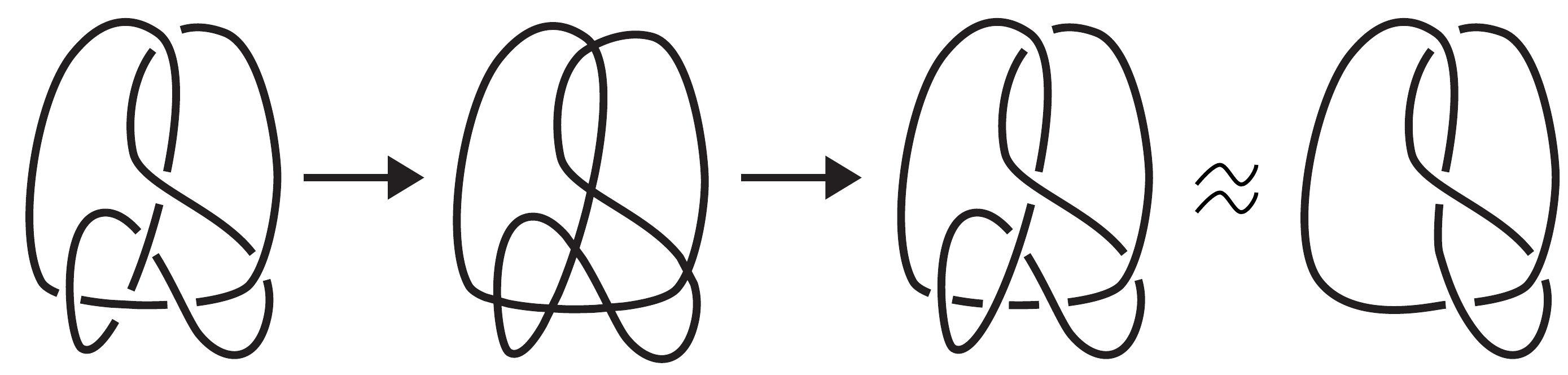} 
\end{center}
\vspace{-.3in}
\caption{Starting with a minimal crossing diagram of $7_6$ (left), we forget crossing information to obtain the diagram's shadow. We then choose all new crossing information in the shadow and find we have a diagram of $4_1$ (right). This illustrates that $7_6$ is a parent of $4_1$.}\label{shadow}
\end{figure} 

We note that a diagram that is missing none, all, or some of its crossing information is referred to as a {\em pseudodiagram}, as in \cite{hanaki}. Unknown crossings in a pseudodiagram are called {\em precrossings.} Equivalence relations for pseudodiagrams, called pseudo-Reidemeister moves, allow us to think of pseudodiagrams in terms of the knots they have the potential to represent. (See \cite{pseudo} to learn much more about these objects.) Pseudodiagrams and pseudoknots will be useful objects in some constructions in Section~\ref{families}.


There are a couple of preliminary results we can state now. For instance, we notice that we can equate knots with their mirror images when proving results about knot lineage, on account of the following general fact.

\begin{proposition}\label{mirror}
If a knot $K$ is a parent of a knot $H$, then $K$ is also a parent of the mirror image of $H$.
\end{proposition}

\begin{proof} Let $P$ be the shadow corresponding to a minimal diagram of knot $K$, and suppose that $H$ is a descendant of $K$. Then we can resolve the precrossings of $P$ to produce $H$. Resolving the precrossings of $P$ in precisely the opposite way produces a diagram of the mirror image of $H$. So $K$ must also be a parent of the mirror image of $H$.
\end{proof}

We also make the following observation related to composite knots.

\begin{proposition}
Suppose that $K_1$ and $K_2$ are either both alternating knots or both torus knots, $H_1$ is a descendant of $K_1$, and $H_2$ is a descendant of $K_2$. Then $H_1\# H_2$ is a descendant of $K_1\# K_2$.
\end{proposition}

\begin{proof} For a pair of alternating knots $K_1$ and $K_2$, we know by \cite{kauffman1, murasugi, thistle} that forming the connect sum of a minimal diagram of $K_1$ with a minimal diagram of $K_2$ in the plane---as in Figure~\ref{comp}---produces a minimal diagram of $K_1\#K_2$. A similar result holds if both $K_1$ and $K_2$ are torus knots by \cite{diao2}. (We note that, in general, forming the connect sum of two minimal crossing diagrams of factor knots may not produce a minimal diagram of the composite knot \cite{lackenby}.) Now, if we resolve a shadow of a minimal diagram of $K_1$ to produce $H_1$ and, likewise, resolve a shadow of a minimal diagram of $K_2$ to produce $H_2$, these diagrams can be connected to produce a diagram of $H_1\# H_2$ that projects to the shadow of a minimal diagram of $K_1\#K_2$.
\end{proof}

\vspace{-.2in}
\begin{figure}[!htbp] 
\begin{center}
\includegraphics[height=1.3in]{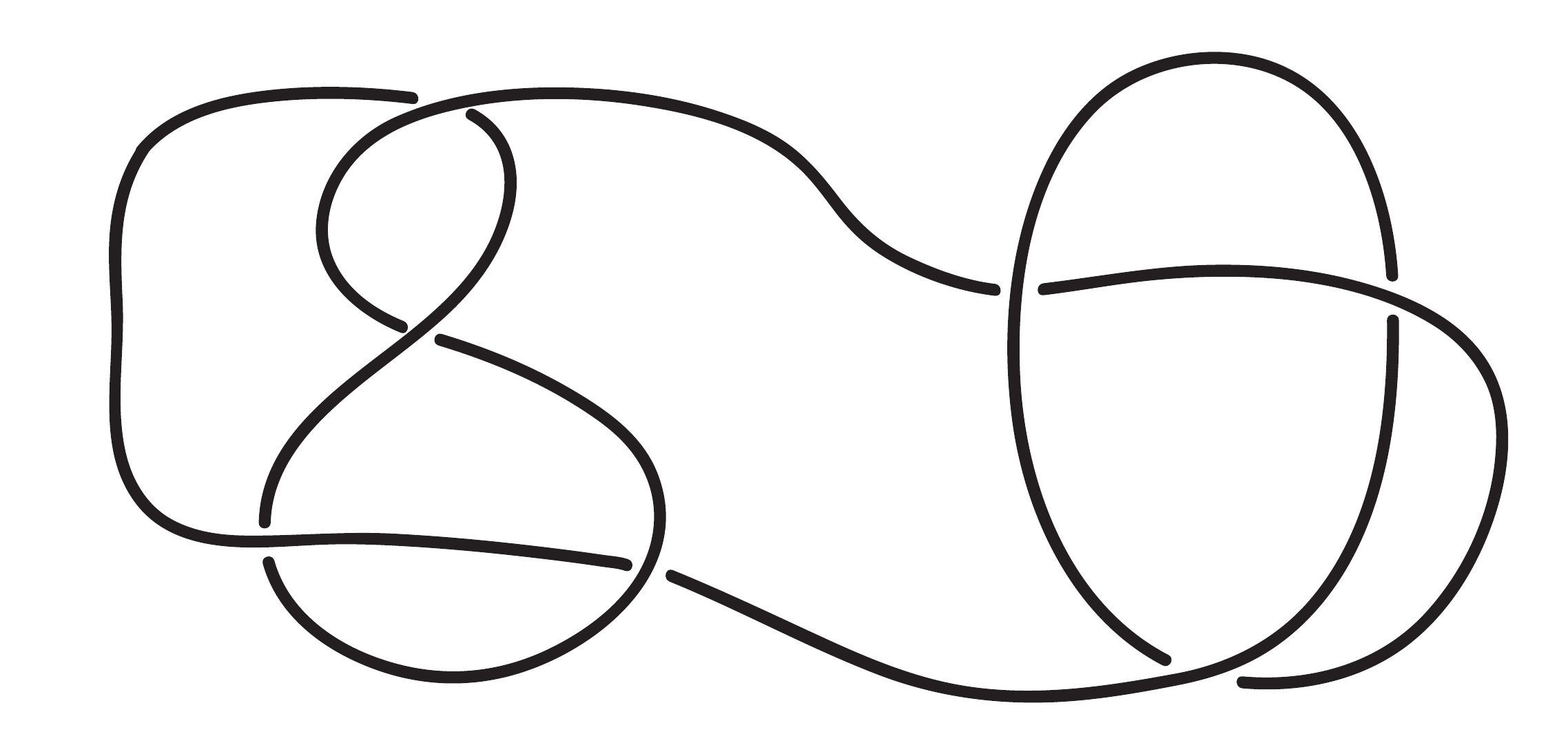} 
\end{center}
\vspace{-.3in}
\caption{The composition of minimal diagrams of $4_1$ and $3_1$.}\label{comp}
\end{figure} 

%


Now, without further ado, let us learn about descendant relations in families of knots.

\section{Families of Knots}\label{families}

In this section, we consider two key knot families: twist knots and $(2,p)$-torus knots. These families distinguish themselves by being closely related to one another---in terms of descendant and parent relations---and rather insular. As a result, we will observe that knots in these families tend to have few descendants.

\begin{figure}[!htbp]
\begin{center}
\includegraphics[width=6.5cm]{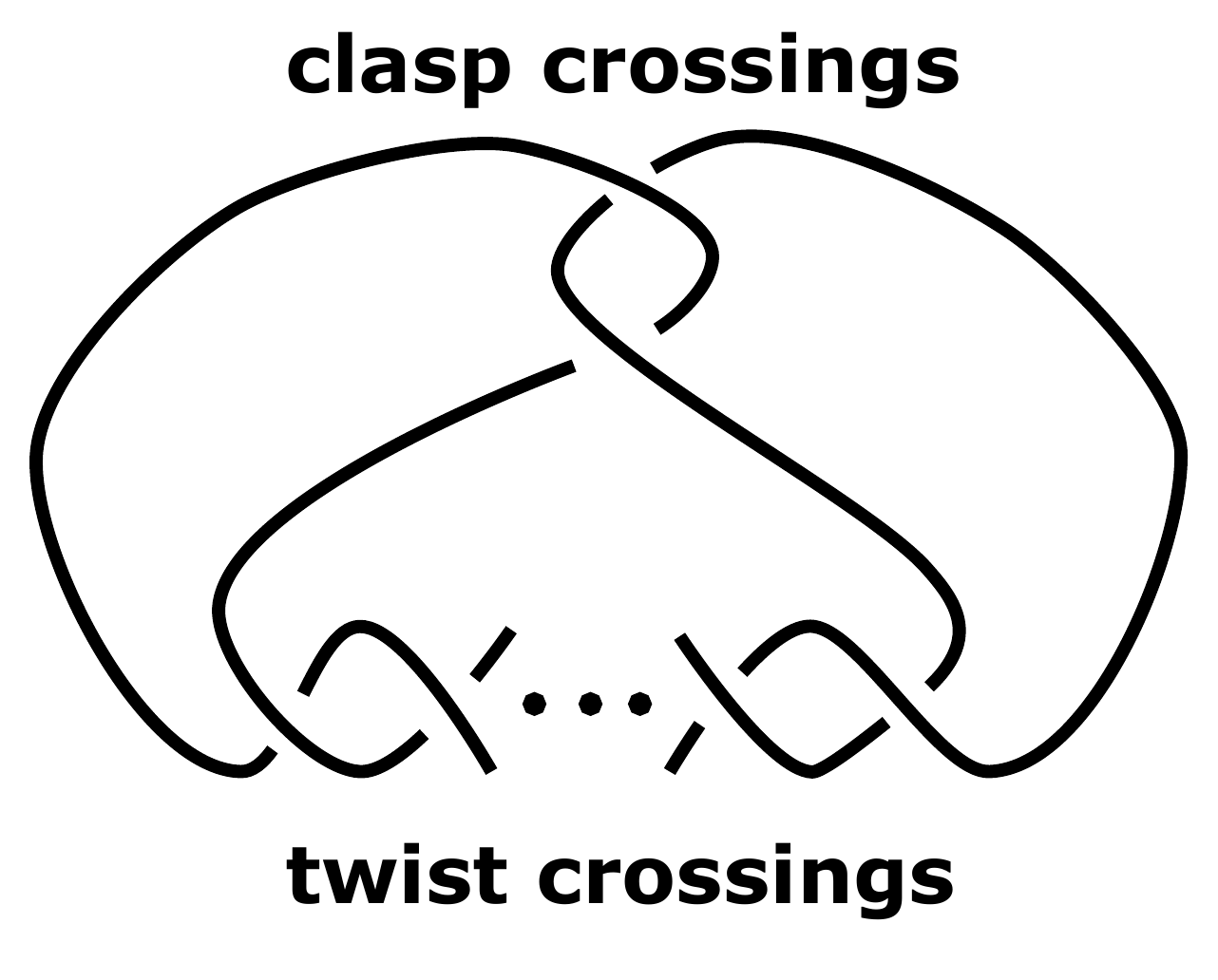}
\end{center}
\vspace{-.2in}
\caption{A generic example of a minimum crossing diagram of a twist knot.}\label{twist_gen}
\end{figure}

Keeping this functional equivalence of knots and their mirror images in mind, let us turn to twist knots. A \textit{twist knot} $T_n$ is a knot formed by two {\em clasp crossings} and $n$ {\em twist crossings}, as in Figure \ref{twist_gen}. We note that all twist knots are alternating, so it suffices to consider their minimal crossing diagrams as in the figure, where the alternating pattern is preserved as we pass {\em between} the clasp crossings and the twist crossings. For these knots, we have the following theorem.

\begin{theorem}\label{twist} The knot $K$ is a descendant of twist knot $T_n$ if and only if $K = T_k$ for some integer $k$ with $0 \leq k \leq n$.  
 \end{theorem}

 \begin{proof}
We begin by proving the ``only if'' direction of Theorem \ref{twist}. Suppose that the knot $K$ is a descendant of twist knot $T_n$ for some positive integer $n$. Then there exists a choice of crossing information for all crossings in the shadow of a standard minimal crossing diagram of $T_n$ (i.e., the shadow of a diagram from Figure~\ref{twist_gen}) that produces a diagram of $K$. Suppose that the clasp crossings, $c_1$ and $c_2$, have opposite signs in this resolution. Then, $K$ is the unknot, $T_{0}$. Otherwise, $c_1$ and $c_2$ are alternating. 

In the remaining $n$ twist crossings, $p$ crossings are positive, and $n-p$ crossings are negative. We can perform $p$ or $n-p$ (whichever number is smaller) Reidemeister 2 moves to make the twist crossings alternating. If $p=n-p$, the result is a two-crossing diagram of the unknot, $T_0$. Otherwise, some number of twist crossings with the same sign remain. In this case, either the entire knot diagram is alternating, or we can remove one crossing---producing a diagram of $T_k$---by using the sequence of Reidemeister moves pictured in Figure \ref{the_end}. 

\begin{figure}
\begin{center}
\includegraphics[width=11cm]{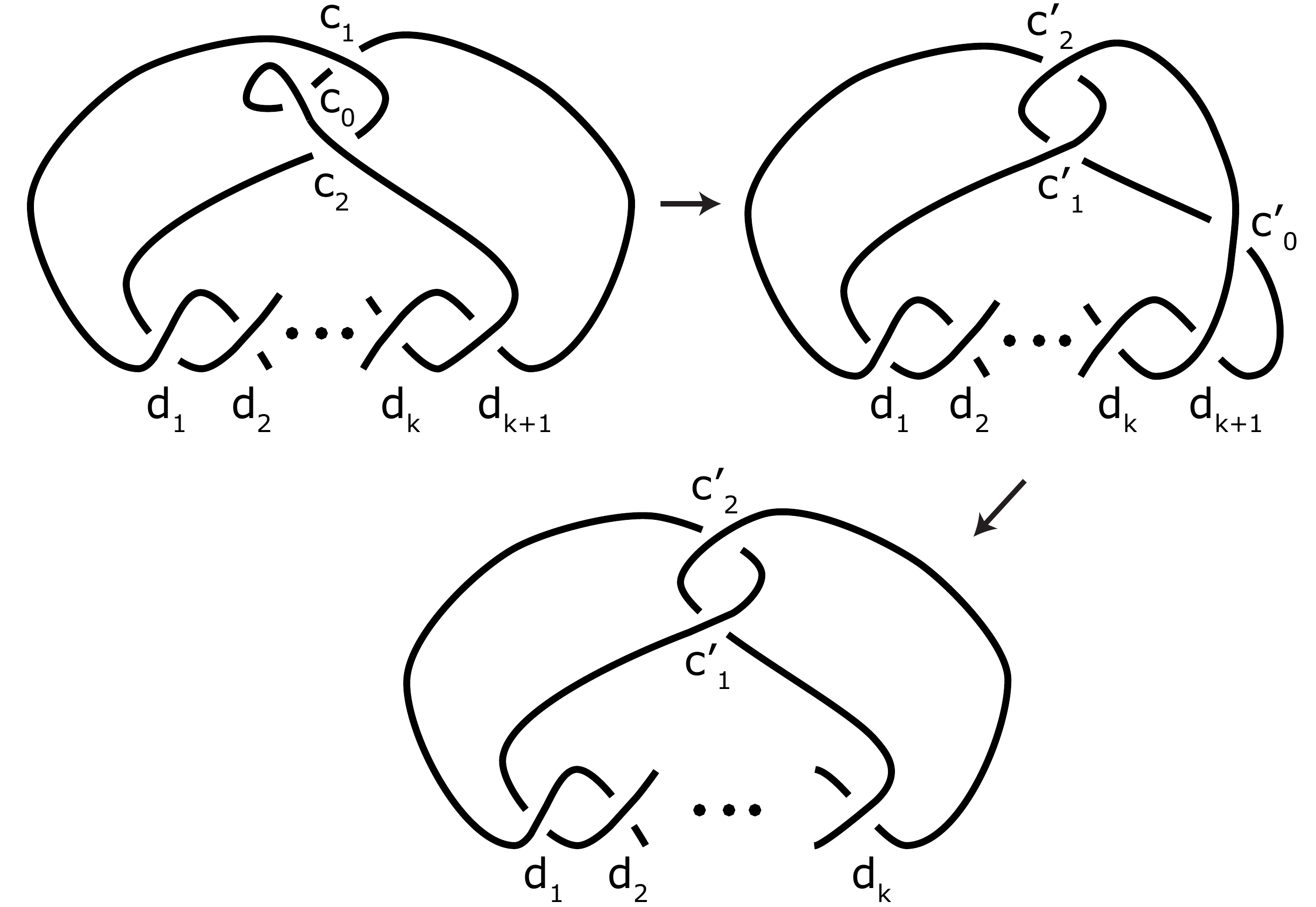}
\end{center}
\vspace{-.1in}
\caption{Simplifying Reidemeister moves}\label{the_end}
\end{figure}
 
To prove the ``if'' direction of the theorem, let $P$ be the shadow of a minimal crossing diagram of $T_n$, and suppose that $k$ is an integer such that $0\leq k \leq n$. For the remainder of this proof, refer to Figure \ref{labetwist} for orientation and labeling conventions.

\begin{figure}
\begin{center}
\includegraphics[width=6.5cm]{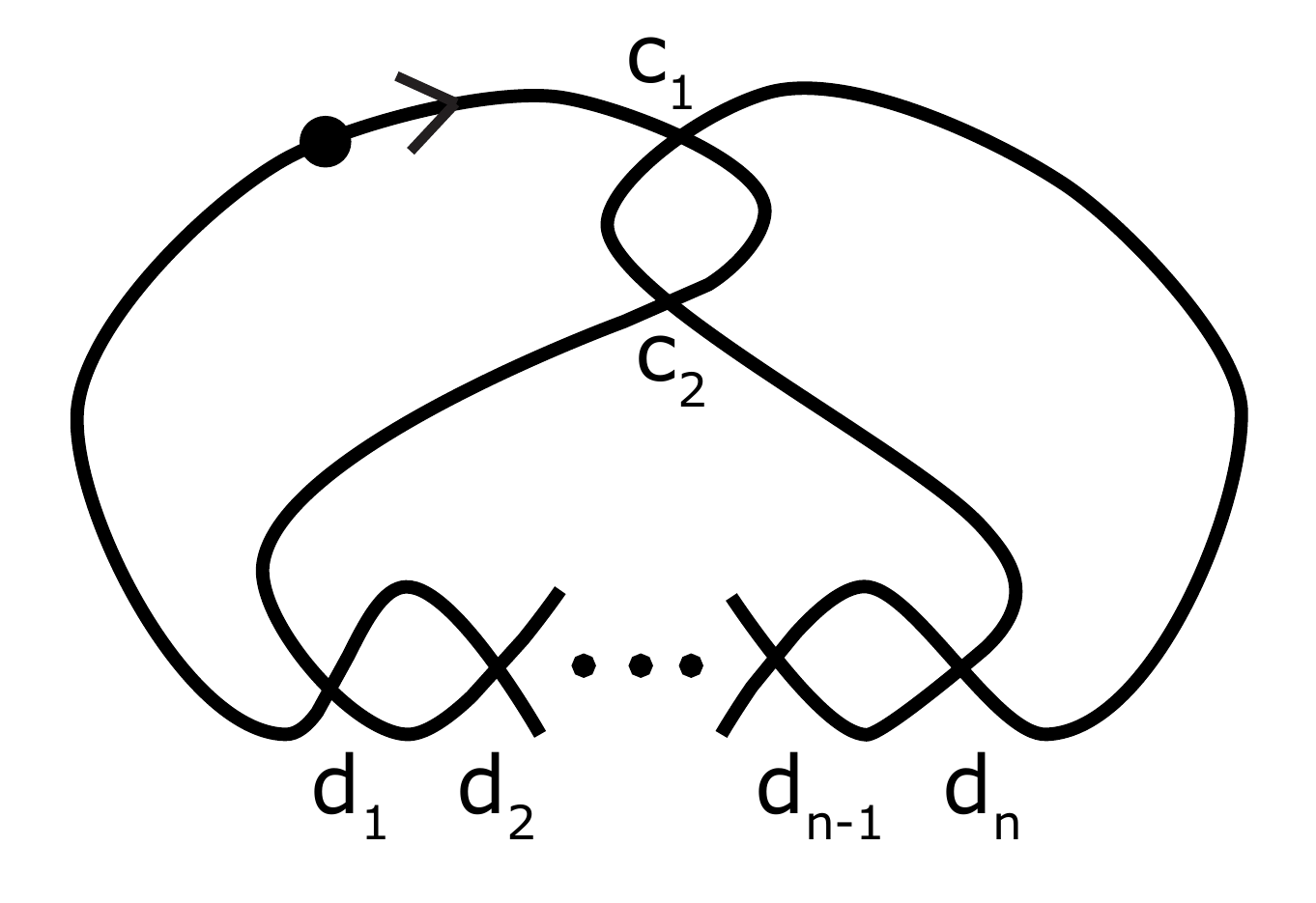}
\end{center}
\vspace{-.2in}
\caption{Twist knot with orientation and labeled precrossings.}\label{labetwist}
\end{figure}

\textit{Case 1: $n-k$ is even.} In the twist of $P$, we resolve the $n-k$ crossings $d_{k+1}, \ d_{k+2},$ $ \cdots, d_{n}$ so that half are positive and half are negative. We then perform $\frac{1}{2}(n-k)$ Reidemeister 2 moves on these crossings to produce the shadow $P'$ of a minimal crossing diagram of $T_k$. Finally, we resolve the precrossings in $P'$ to produce a diagram of $T_k$.

\textit{Case 2: $n-k$ is odd.} We resolve the $n-k-1$ precrossings, $d_{k+2}, \ d_{k+3},  \cdots , d_{n},$ in the twist of $P$ as in Case 1 so that half of the crossings are positive and half of the crossings are negative. Then, we perform $\frac{1}{2}(n-k-1)$ Reidemeister 2 moves to eliminate crossings $d_{k+2}, \ d_{k+3}, \ \cdots , d_{n}$. 

Next, if $k+1$ is odd, we resolve crossings $c_1$ and $c_2$ to be negative and $d_1, \cdots, d_{k+1}$ to be positive. If $k+1$ is even, we resolve all of the crossings in the shadow so they are positive. In both cases, the effect of resolving precrossings in this way puts us in the situation of Figure \ref{the_end} where we obtain a minimal crossing diagram of $T_{k}$ after a sequence of Reidemeister moves.
\end{proof}

Next, we prove a similar result for another knot family: the $T_{2,p}$ torus knots, i.e., the knots that can be represented as closures of 2-braids.  Note that, since these torus knots are alternating, it suffices to consider their standard minimal diagrams as in, for example, Figure~\ref{torus_r2}.

\begin{theorem}\label{2ptorus}
The knot $K$ is a descendant of torus knot $T_{2,p}$ if and only if $K = T_{2, q}$ for some $0<q \leq p$, where $p$ and $q$ are odd integers. 
\end{theorem}

\begin{proof}
Let $K$ be a descendant of torus knot $T_{2,p}$ for some positive odd integer $p$. Then, there exists a choice of crossing resolutions that produces a diagram of $K$ from the shadow of a minimal crossing diagram of $T_{2,p}$. Suppose there are $n$ positive crossing resolutions and $m$ negative crossing resolutions for some non-negative integers $n$ and $m$. If we let $l=$ min$\{n,m\}$, we can perform $l$ Reidemeister 2 moves to reduce the number of crossings in the diagram by $2l$. The resulting knot is $T_{2, p-2l}$.
 
To prove the converse, we consider the torus knot $T_{2,p}$ where $p$ is a positive odd integer, and we let $q$ be a positive odd integer less than $p$. In the shadow of $T_{2,p}$, we resolve $q+\frac{p-q}{2}$ crossings to be positive and the remaining $\frac{p-q}{2}$ crossings to be negative. This enables us to perform $\frac{p-q}{2}$ Reidemeister 2 moves---as in the Figure \ref{torus_r2} example---to obtain a minimal crossing diagram of $T_{2,q}$.  \end{proof}

\begin{figure}[!htbp]
\begin{center}
\includegraphics[width=10cm]{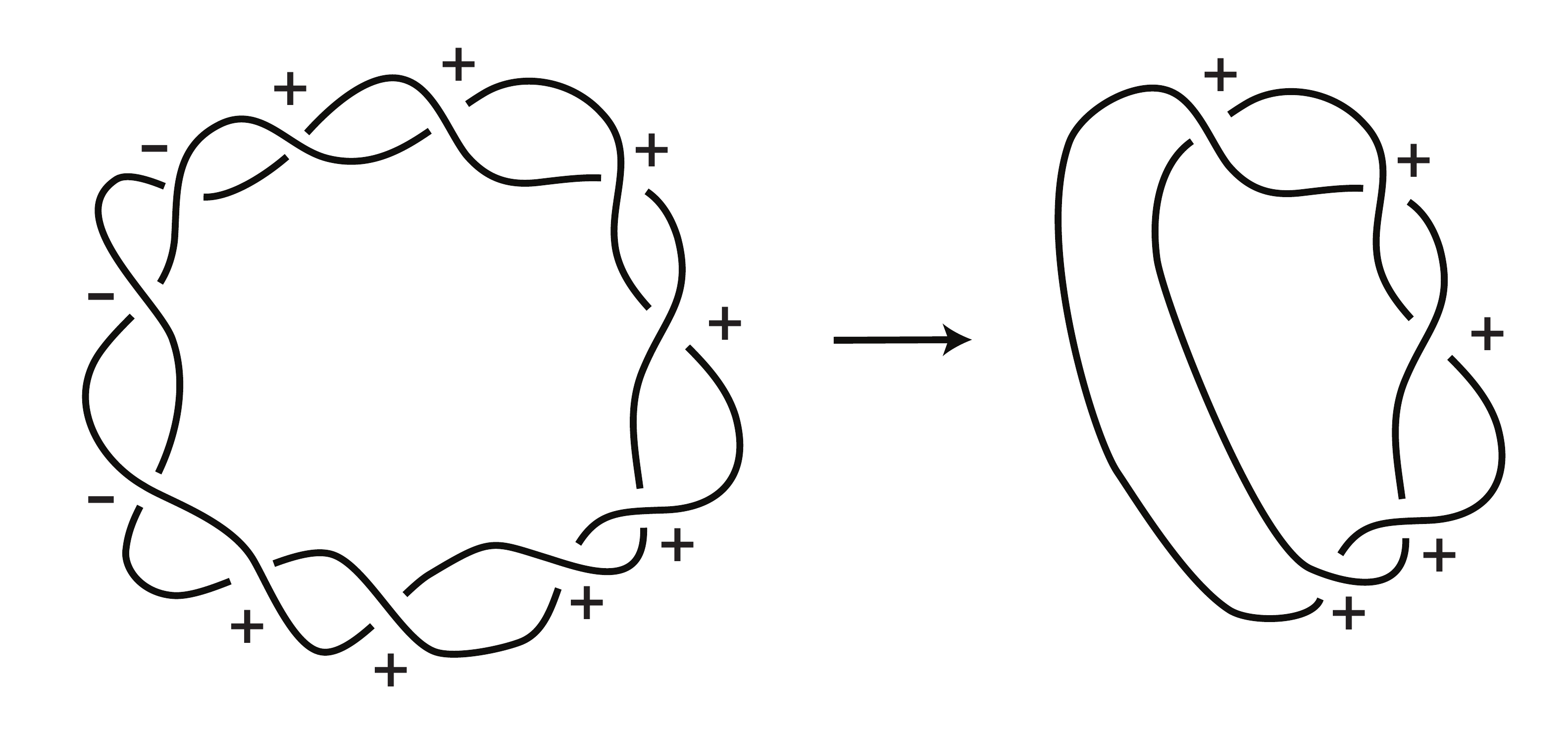}
\end{center}
\vspace{-.2in}
\caption{Resolving crossings in the $T_{2,11}$ shadow to produce a minimal crossing diagram of $T_{2,5}$ after three R2 moves.}\label{torus_r2}
\end{figure}

\section{Notions of Fertility}\label{n-fertile}

As we mentioned in the introduction, $0_1$, $3_1$, $4_1$, $5_2$, $6_2$, $6_3$, and $7_6$ are all fertile knots. On the other hand, the results of Section 2 allow us to show that knots $5_1$, $6_1$, $7_1$, and $7_2$ are not fertile. Indeed, the knot $5_1$ is $T_{2,5}$ and $7_1$ is $T_{2,7}$, while the figure-eight knot, $4_1$ does not equal $T_{2,p}$ for any $p$, so $4_1$ fails to be a descendant of either knot. Furthermore, $6_1$ and $7_2$ are both twist knots while $5_1$ is not a twist knot and, hence, not a descendant. So both $6_1$ and $7_2$ fail to be fertile. A fun afternoon exercise shows that all other knots with seven or fewer crossings---namely, $7_3$, $7_4$, $7_5$, and $7_7$---fail to be fertile. For instance, we can show that $6_1$ is not a descendant of $7_3$ or $7_5$ while $6_3$ is not a descendant of $7_4$ or $7_7$.

What about prime knots with more than seven crossings? What about composite knots? Computations for all prime and composite knots with 10 or fewer crossings reveals that there are no 8-, 9-, or 10-crossing fertile knots. In fact, we conjecture that there are no fertile knots with crossing number greater than seven. 

So, is fertility a helpful notion? Yes and no. To gain insights about a knot's structure from the notion of fertility, we expand our definition.

\begin{definition} A knot is {\bf \em n}-{\bf fertile} if it is a parent of every knot with $n$ or fewer crossings.
\end{definition}

In other words, if a knot $K$ is $n$-fertile and has crossing number $cr(K)=n+1$, then $K$ is fertile. But for every knot $K$, there is some nonnegative integer $n$ for which $K$ is $n$-fertile. This notion enables us to define a new knot invariant.

\begin{definition} The {\bf fertility number}, $F(K)$, of a knot $K$ is the greatest integer $n$ for which $K$ is $n$-fertile. 
\end{definition}

What do we already know about the fertility number? First, the maximum value of $F(K)$ for a knot with crossing number $cr(K)>4$ is $cr(K)-1$. (Notice that the figure-eight knot is 4-fertile, the trefoil is 3-fertile, and the unknot is 0-fertile.) Furthermore, 0 is clearly a lower bound for $F(K)$, but perhaps we can say more. In Section~\ref{opp}, we will discuss how to improve the lower bound of $F(K)$ in general. For small knots, the results in Section~\ref{computations} will tell us so much more. 

Before we look at fertility numbers for specific knots, we introduce one more notion related to knot fertility called $(n,m)$-fertility. 

\begin{definition}\label{nm} A knot $K$ is {\bf ({\em n,m})}-{\bf fertile} if for each prime knot $H$ with $cr(H)\leq n$, there is a knot shadow (not necessarily from a minimal diagram) with $m$ crossings that contains both $H$ and $K$.\end{definition}

To determine if a knot $K$ is $(n,m)$-fertile, one searches through all $m$-crossing shadows that have $K$ as a resolution. (For instance, if $m=cr(K)$, then this set of
shadows is just all shadows coming from minimal diagrams of $K$.)  If the union of the resolution sets of all these shadows contains all knots with $n$ or fewer crossings, then $K$ is $(n,m)$-fertile. 

One of the first things we notice from Definition~\ref{nm} is that $(n-1,n)$-fertility is equivalent to fertility for $n$ crossing knots. We explore the concept of $(n,m)$-fertility further in Sections~\ref{computations} and \ref{opp}.

\section{Computational Results}\label{computations}

Cantarella et al.~\cite{knotdiagrams} have computed an exhaustive list of all
knot shadows through 10 crossings.  By assigning all possible crossings to
these shadows and analyzing all knot types resulting from the
resulting knot diagrams, we obtain complete information about knot
fertility for knots through 10 crossings.  

The knot types were computed using \texttt{lmpoly}, a program by Ewing
and Millett \cite{millettewing} to compute the HOMFLYPT polynomial
\cite{HOMFLY} of diagrams, and \texttt{knotfind}, a program to compute
knot types from Dowker codes, which is a part of the larger program
Knotscape by Hoste and Thistlethwaite \cite{knotscape}.  The program
\texttt{knotfind} does not distinguish between the different
chiralities of chiral knots.  Since fertility is a non-chiral
property (see Prop. \ref{mirror}), \texttt{knotfind} would have been sufficient for the work
here.  However, we used the HOMFLYPT computation via \texttt{lmpoly}
as a further check on the data.

In Table \ref{shadowcountstable}, we show for each crossing number the total number of shadows, the number of
shadows which host only unknots (which we call {\em totally unknotted}), and the number of shadows which yield
a minimal crossing diagram (partitioned into prime and composite knot shadows).
Note that some shadows host minimal diagrams for multiple
knot types. In the table, the number of minimal diagrams are counted
without multiplicity.

\begin{table}
  \caption{Counts for the number of shadows, number of totally unknotted shadows, and number of shadows corresponding to minimal diagrams by crossing number.}
  \medskip
  
  \centering
\begin{tabular}{c||c|c|c|c}
  crossings & shadows & totally unknotted & minimal prime & minimal composite \\
  \hline
  3 & 6 & 5 & 1 & 0\\
  4 & 19 & 16 & 1 &0 \\
  5 & 76 & 55 & 2 & 0\\
  6 & 376 & 240 & 3 & 2 \\
  7 & 2194 & 1149 & 10 & 3 \\
  8 & 14614 & 6229 & 27 & 13 \\
  9 & 106421 & 35995 & 101 & 59 \\
  10 & 823832 & 219272 & 364 & 263
\end{tabular}
\label{shadowcountstable}
\end{table}

When different minimal diagrams exist for alternating knot types, they
are all related by flype moves \cite{MWTetal}.  Thus, the list of descendants derived from one minimal diagram of an alternating knot is precisely the same as the list derived from any other minimal diagram.  
Minimal diagrams of non-alternating knot types, though, need not be related by flype moves. As a result, the lists of
descendants derived from different minimal diagrams of a non-alternating knot might differ.  In our fertility
table, Table \ref{fertilitytable}, we considered a knot type $H$ to be
a descendant of a non-alternating knot type $K$ if there exists a shadow containing $n=cr(K)$ crossings
that has both $K$ and $H$ as resolutions. This is consistent with---though it represents a different way of viewing---our original definition of ``descendant." 

To give the reader some idea of the computational complexity of this problem, we provide Table \ref{nonalttable}, which shows the number of minimal diagrams for non-alternating knot types through 10 crossings.

\begin{table}[h]
  \caption{Numbers of minimal diagrams for non-alternating prime knot types}
  \medskip
  
  \centering
\begin{tabular}{c|c||c|c||c|c||c|c||c|c}
  knot & shadows & knot & shadows & knot & shadows & knot & shadows & knot & shadows\\
  \hline
  $8_{19}$ & 5  & $10_{124}$ & 13 & $10_{135}$ & 21 & $10_{146}$ & 17 & $10_{157}$ & 2  \\
  $8_{20}$ & 7  & $10_{125}$ & 12 & $10_{136}$ & 26 & $10_{147}$ & 18 & $10_{158}$ & 3  \\
  $8_{21}$ & 5  & $10_{126}$ & 8  & $10_{137}$ & 27 & $10_{148}$ & 11 & $10_{159}$ & 5  \\ 
  $9_{42}$ & 17 & $10_{127}$ & 7  & $10_{138}$ & 19 & $10_{149}$ & 7  & $10_{160}$ & 10 \\ 
  $9_{43}$ & 14 & $10_{128}$ & 20 & $10_{139}$ & 5  & $10_{150}$ & 19 & $10_{161}$ & 8  \\
  $9_{44}$ & 17 & $10_{129}$ & 19 & $10_{140}$ & 16 & $10_{151}$ & 15 & $10_{162}$ & 3  \\
  $9_{45}$ & 14 & $10_{130}$ & 14 & $10_{141}$ & 11 & $10_{152}$ & 1  & $10_{163}$ & 3  \\
  $9_{46}$ & 8  & $10_{131}$ & 14 & $10_{142}$ & 8  & $10_{153}$ & 2  & $10_{164}$ & 3  \\
  $9_{47}$ & 2  & $10_{132}$ & 35 & $10_{143}$ & 13 & $10_{154}$ & 3  & $10_{165}$ & 3  \\
  $9_{48}$ & 4  & $10_{133}$ & 32 & $10_{144}$ & 8  & $10_{155}$ & 6  & \\
  $9_{49}$ & 2  & $10_{134}$ & 20 & $10_{145}$ & 10 & $10_{156}$ & 6  & \\
\end{tabular}
\label{nonalttable}
\end{table}

In Table \ref{fertilitytable}, we record fertility numbers, which provide one sort of summary of our more significant computational results. In fact, we derived an exhaustive list of descendants of all knots with up to 10 crossings. This collection of data is too vast to share here, but we can share some interesting observations we derived from the data. 

First, and perhaps most interestingly, we found the descendant relation is decidedly {\em not} transitive. There are a number of counterexamples
to transitivity, most of which involve non-alternating knot types. However, there are nine examples of non-transitivity involving only
alternating knot types through 10 crossings. See Table~\ref{nontransitive}.

\begin{table}[h]
  \caption{Triples of alternating knots such that $K_1$ is a descendant of $K_2$ and $K_2$ is a descendant of $K_3$, but $K_1$ is not a descendant of $K_3$.}
  \medskip
  
  \centering
\begin{tabular}{c|c|c}
  $K_1$ & $K_2$ & $K_3$\\
  \hline
$7_1$ & $8_9$ & $10_{123}$\\
$7_2$ & $8_{14}$ & $10_{96}$\\
$7_3$ & $8_{7}$ & $10_{116}$\\
$7_3$ & $8_{11}$ & $10_{92}$\\
$7_3$ & $8_{11}$ & $10_{113}$\\
$7_4$ & $8_{11}$ & $10_{92}$\\
$7_4$ & $8_{11}$ & $10_{113}$\\
$7_5$ & $8_{9}$ & $10_{123}$\\
$7_5$ & $8_{14}$ & $10_{117}$\\

\end{tabular}
\label{nontransitive}
\end{table}

The use of the term ``descendant'' in this paper leads to other strange naming conventions for knot relationships.
In particular, we call a knot $K_1$ a \textit{sibling} of a knot $K_2$, where $n=cr(K_1)=cr(K_2)$, if there is an $n$-crossing shadow that
has both $K_1$ and $K_2$ as resolutions.\footnote{The term ``sibling" was used in a similar context in~\cite{diao}, though the two definitions differ.}  Note that for a given shadow, there are only two ways of assigning crossings
to yield an alternating diagram, and these two diagrams are mirror
images.  Thus, these sibling relationships only occur between
alternating and non-alternating or non-alternating and non-alternating
knot types.  Table \ref{siblingstable} shows all sibling relationships
for 8- and 9-crossing knot types.  We created a table of sibling relationships between 10-crossing knots, but it is too unwieldy to include here. For instance, $10_{132}$ alone has 41 siblings!

Another interesting aspect of $n$-fertility that can be studied is a notion of {\em anti-fertility}. Specifically, which knots act as roadblocks to more complex knots achieving $n$-fertility for various values of $n$? For instance, our computations show that the knots $7_1$, $7_4$, and $7_7$ fail to be descendants of 15 of the 18 alternating 8-crossing knots. In essence, these three knots are providing a significant barrier to 8-crossing knots achieving fertility. These three knots continue to be problematic for alternating knots with crossing number 9 or 10. Of the 41 alternating 9-crossing knots, $7_1$ fails to be a descendant of 31 knots, $7_4$ fails to be a descendant of 26 knots, and $7_7$ fails to be a descendant of 19 knots. Similarly, of the 123 alternating 10-crossing knots, $7_1$ fails to be a descendant of 77 knots, $7_4$ fails to be a descendant of 67 knots, and $7_7$ fails to be a descendant of 53 knots. 

Consider another unusual example. If we examine the shadows of alternating 9-crossing knots, we find that $8_{16}$ and $8_{17}$ fail to appear in the resolution sets of of 40 of the 41 representative shadows.\footnote{Note that we are only considering the shadow of one representative from each equivalence class of reduced, alternating knot diagrams that are related by flypes.}  In other words, $8_{16}$ and $8_{17}$ are each only descendants of one alternating 9-crossing knot ($9_{33}$ and $9_{32}$, respectively). The $8_3$ knot provides another barrier to fertility for 9- and 10-crossing knots. It fails to be a descendant of 38 out of 41 alternating 9-crossing knots and 109 of 123 alternating 10-crossing knots.

Just as we can make a number of observations about $n$-fertility and anti-fertility, the data set helps us compute $(n,m)$-fertility numbers for a number of knots. See Tables \ref{m3fertilitytable}-\ref{m10fertilitytable} in the appendix for a complete list of $(n,m)$-fertility numbers for knots with up to 10 crossings. We highlight a subset of this data here, namely examples of knots that are $(k,k)$-fertile for some $k$. Table~\ref{kkfertile} shows a complete list of $(k,k)$-fertile knots for $6\leq k\leq 10$. It is no surprise that the unknot is $(k,k)$ fertile for each $k$ since every knot shadow can be resolved to produce a diagram of the unknot. More interestingly, we notice that $3_1$, $4_1$, and $5_2$ are $(k,k)$-fertile for certain $k$ values, following regular patterns. Using the observations of Table~\ref{kktable} and the results of Section~\ref{opp}, we are able to prove that these patterns hold {\em ad infinitum}.

\begin{table}[h]
  \caption{A  summary of knots that are $(k,k)$-fertile}\label{kktable}
  \medskip
  
  \centering
\begin{tabular}{c|l}
  $k$& Knots that are $(k,k)$-fertile\\
  \hline
 6 & $0_1$, $3_1$, $4_1$, $5_2$\\
 7 & $0_1$, $3_1$\\
 8 & $0_1$, $3_1$, $4_1$, $5_2$\\
 9 & $0_1$, $3_1$\\
10 & $0_1$, $3_1$, $4_1$, $5_2$\\
\end{tabular}
\label{kkfertile}
\end{table}

There are likely many more interesting observations that can be made from staring at our data, but for now, we move on to discussing how our work fits in the context of the work done by others over the years.

\section{Other Knot Relations}\label{opp}

Our notion of descendant is not the first notion that attempts to capture what it means for a knot to contain another knot. For instance, the descendant relation shares some similarities with the notion of {\em predecessor}, introduced in \cite{diao}. 

\begin{definition} A knot $K_1$ is a {\bf predecessor} of knot $K_2$ if the crossing number of $K_1$ is less than that of $K_2$ and $K_1$ can be obtained from a minimal diagram of $K_2$ by a single crossing change.
\end{definition}

We see that, while the trefoil is a descendant of the figure-eight knot, it is not a predecessor since two crossing changes are required to derive a trefoil from a minimal diagram of the figure-eight knot. So if a knot $K_1$ is a predecessor of a knot $K_2$, then $K_1$ is a descendant of $K_2$, but the converse fails to hold in general.

In another effort to relate knots to their parts, Millett and Jablan described what it means for one knot to be a {\em subknot} of another knot~\cite{millett1}. In particular, they gave the following definition.

\begin{definition} A knot $K_1$ is a {\bf subknot} of knot $K_2$ if it can be obtained from a minimal diagram of $K_2$ by crossing changes that preserve those within a segment of the diagram and change those outside this segment so that the complementary segment is strictly ascending.
\end{definition}

\begin{figure}[!htbp] 
\begin{center}
\includegraphics[height=1.1in]{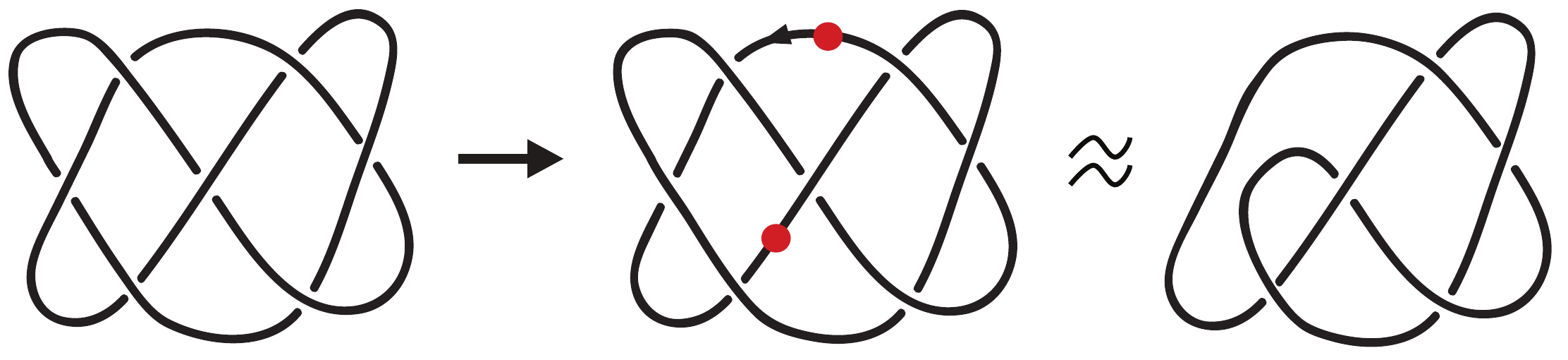} 
\end{center}
\vspace{-.3in}
\caption{The knot $5_2$ (right) is a subknot of the knot $7_4$ (left).}\label{subknot}
\end{figure}

In Figure~\ref{subknot}, we see that the twist knot $5_2$ is a subknot of $7_4$. Indeed, if crossing changes are performed on a segment of the knot $7_4$ between the two dots so that this segment is ascending, we find that knot $5_2$ is produced. 

While the subknot definition {\em seems} stricter than the descendant relation, is it possible that the subknot and descendant relations are actually the same relations in disguise? Just as with the predecessor relation, we immediately observe that if a knot $K_1$ is a subknot of knot $K_2$, then $K_1$ is a descendant of $K_2$. What about the converse? In \cite{millett2}, it was proven that the trefoil is {\em not} a subknot of $11a135$. On the other hand, we illustrated in Figure~\ref{parent_ex} that the trefoil is a descendant of $11a135$, so the descendant and subknot relations must be distinct.

Finally, another well-known relationship between knots is one that was introduced by Taniyama in \cite{taniyama}. 

\begin{definition} A knot $K_1$ is a {\bf minor} of knot $K_2$ if the set of knot shadows that have $K_2$ as a resolution is a subset of the set of shadows that have $K_1$ as a resolution.
\end{definition}

Once again, we notice that if $K_1$ is a minor of $K_2$, then $K_1$ is a descendant of $K_2$. Since all shadows of diagrams of $K_2$ are also shadows of diagrams of $K_1$ when $K_1$ is a minor of $K_2$, then in particular, a minimal crossing shadow of $K_2$ will have $K_1$ as a resolution. To determine if the minor relation is the same as the descendant relation, we need to ask about the converse. Our results from Section~\ref{computations} hold the key to answering this question. Our computations show that $7_3$ is a descendant of $8_{11}$. We also know that $8_{11}$ is a descendant of $10_{92}$, but $7_3$ is not a descendant of $10_{92}$---in fact, this trio of knots is one of our counterexamples to transitivity. More concretely, the shadow in Figure~\ref{minor_ex} contains $8_{11}$ but not $7_3$. So $7_3$ is not a minor of $8_{11}$. This proves that the descendant and minor relations are distinct.

\begin{figure}[!htbp] 
\begin{center}
\includegraphics[height=1.5in]{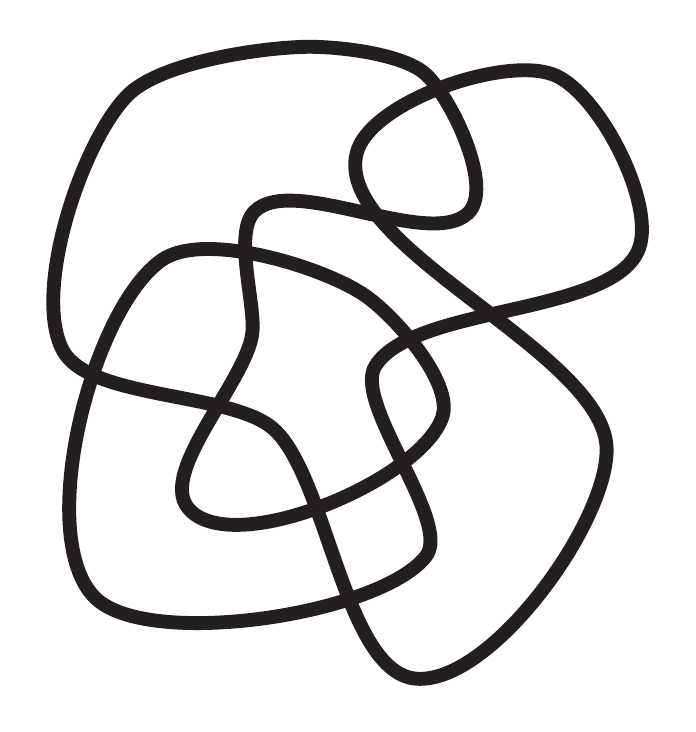} 
\end{center}
\vspace{-.3in}
\caption{A shadow of $10_{92}$ that contains $8_{11}$ but not $7_3$.}\label{minor_ex}
\end{figure}

While the descendant and minor relations are interestingly different, the fact that if $K_1$ is a minor of $K_2$, then $K_1$ is a descendant of $K_2$ yields the following result as a corollary of Theorem 1 in \cite{taniyama}.

\begin{corollary}\label{tref} Every nontrivial knot has the trefoil as a descendant.
\end{corollary}

From this result, we see in particular that the fertility number $F(K)$ for any nontrivial knot $K$ must be at least 3. We can also infer the following result. 

\begin{theorem} The trefoil is $(k,k)$-fertile for all $k\geq 3$. \end{theorem}

Similarly, we have the following three corollaries of Theorems 2, 3, and 4 in \cite{taniyama}. Each of these corollaries has further implications for bounds on fertility numbers.

\begin{corollary}\label{4_1fertility} If $K$ has a prime factor that is not equivalent to a $(2,p)$-torus knot with $p\geq 3$, then $K$ has the figure-eight knot as a descendant.
\end{corollary}

\begin{corollary}\label{5_2fertility} If $K$ has a prime factor that is not equivalent to a $(2,p)$-torus knot with $p\geq 3$ or the figure-eight knot, then $K$ has the knot $5_2$ as a descendant.
\end{corollary}

\begin{corollary} If $K$ has a prime factor that is not equivalent to any of the pretzel knots $L(p_1,p_2,p_3)$ (where $p_1$, $p_2$, and $p_3$ are odd integers), then $K$ has the knot $5_1$ as a descendant.
\end{corollary}

Corollaries~\ref{4_1fertility} and \ref{5_2fertility} have a particularly nice consequence for $(k,k)$-fertility that puts the data in Table \ref{kktable} into a larger context. We can use these corollaries to prove the following result.

\begin{theorem} The figure-eight knot, $4_1$, and the $T_3$ twist knot, $5_2$, are both $(2j,2j)$-fertile for any integer $j\geq 3$.
\end{theorem}

\begin{proof} 
Let us fix an integer $j\geq 3$. Our goal is to show: for every
prime knot with
crossing number less than or equal to $2j$, there is a knot shadow
with $2j$ crossings that has
$4_1$ (resp. $5_2$) as a descendant.  We prove the result for $4_1$.
The proof for $5_2$ is similar.

Let $K$ be any knot with $cr(K)\leq 2j$.  By Corollary \ref{4_1fertility}, the only prime knots
that fail to have $4_1$ as a descendant are $(2,p)$ torus knots.

Case 1: $K$ is not a $(2,p)$ torus knot.
We know that $4_1$ is a descendant of $K$, i.e. there is a
$cr(K)$-crossing shadow $S$ that can
be resolved to both $4_1$ and $K$.
If $cr(K)=2j$, then we are done.  If $cr(K)<2j$, then we can add
$n=2j-cr(K)$ crossings by applying pseudo-R1 moves to $S$ to
obtain a $2j$-crossing shadow which resolves to both $4_1$ and $K$.

Case 2: $K$ is a $(2,p)$ torus knot.
In Figure 11, we provide a $2j$-crossing shadow that can be resolved
to produce $4_1$, $5_2$, and all
$(2,p)$ torus knots with $p$ odd, $0<p<2j$, which suffices to prove
the theorem. We illustrate in subdiagram (b) that the $(2,2j-1)$ torus knot can be derived from the shadow pictured in (a). By an argument similar to the one that proved Theorem~\ref{2ptorus}, we can also derive from this shadow every $(2,p)$ torus knot with $0<p<2j-1$, where $p$ is an odd integer. Subdiagrams (c) and (d) illustrate our desired resolutions of the $4_1$ and $5_2$ knots, respectively. 
\end{proof}

\begin{figure}[!htbp] 
\begin{center}
\includegraphics[height=1.8in]{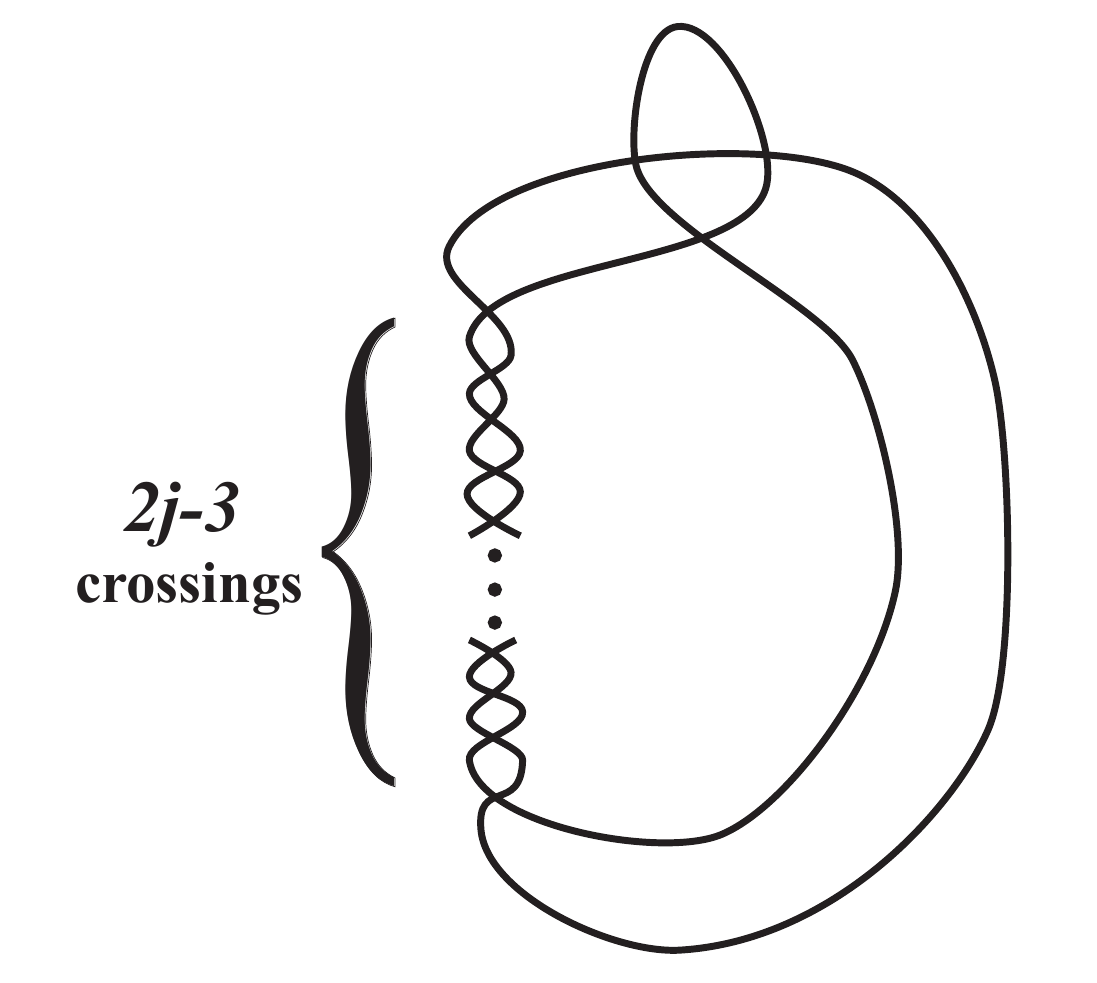}\\
\hspace{.5in}(a)\\
\medskip

\includegraphics[height=1.75in]{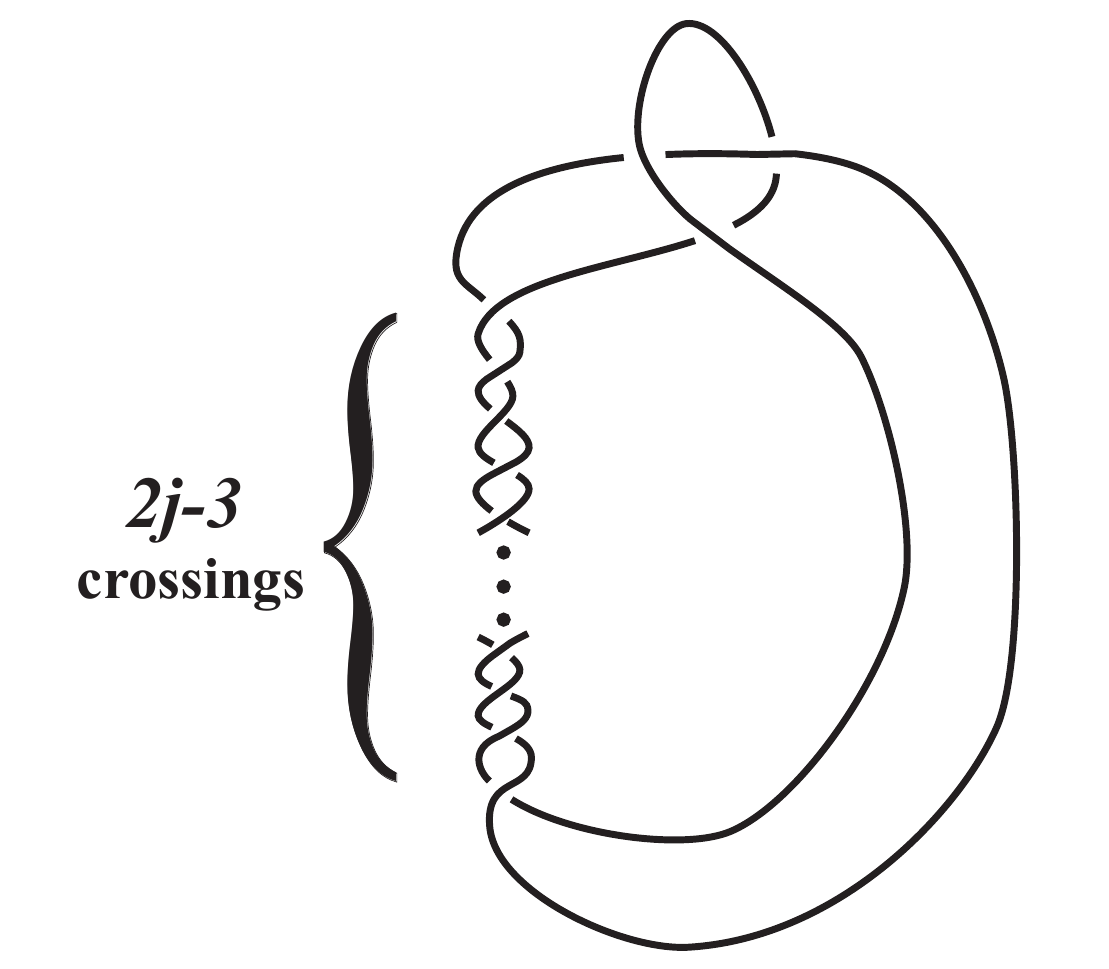} 
\includegraphics[height=1.8in]{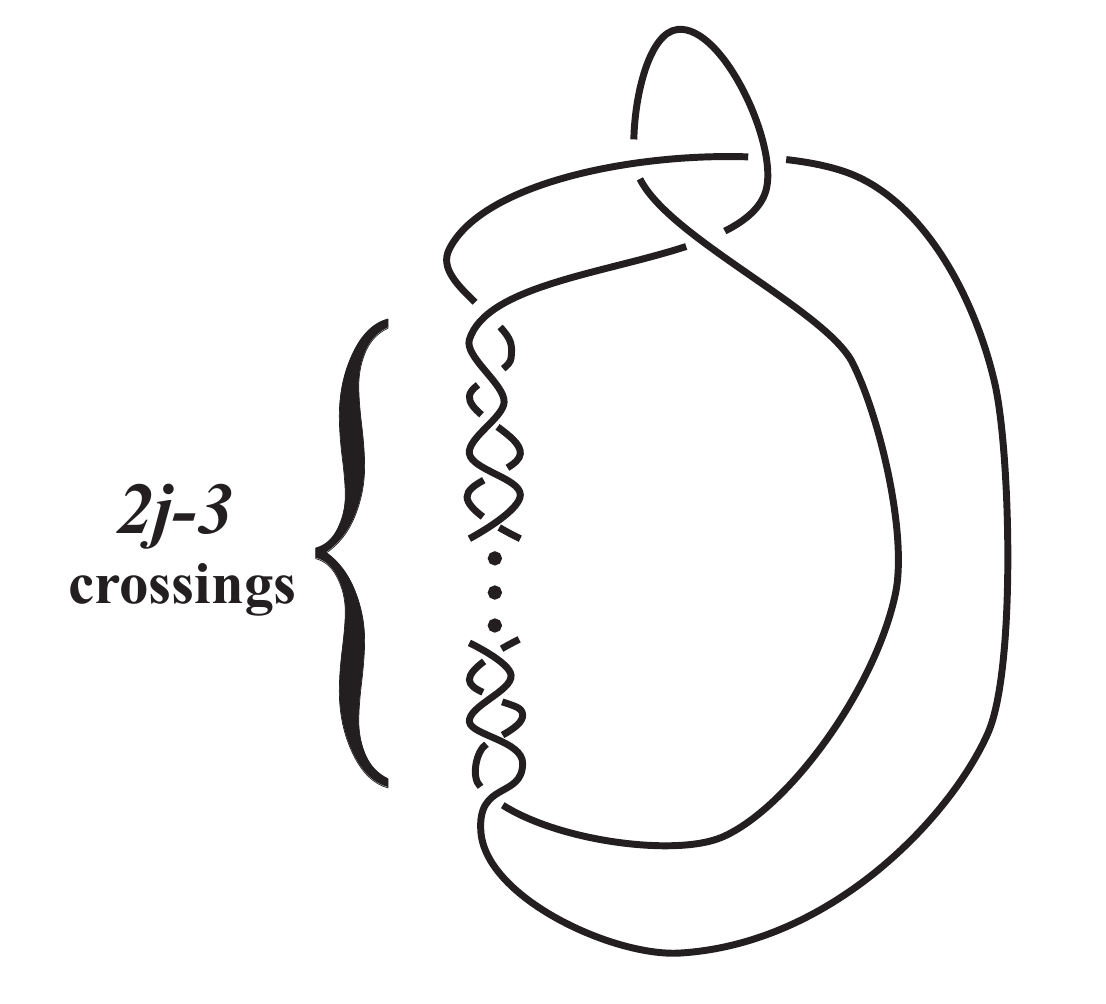}\includegraphics[height=1.8in]{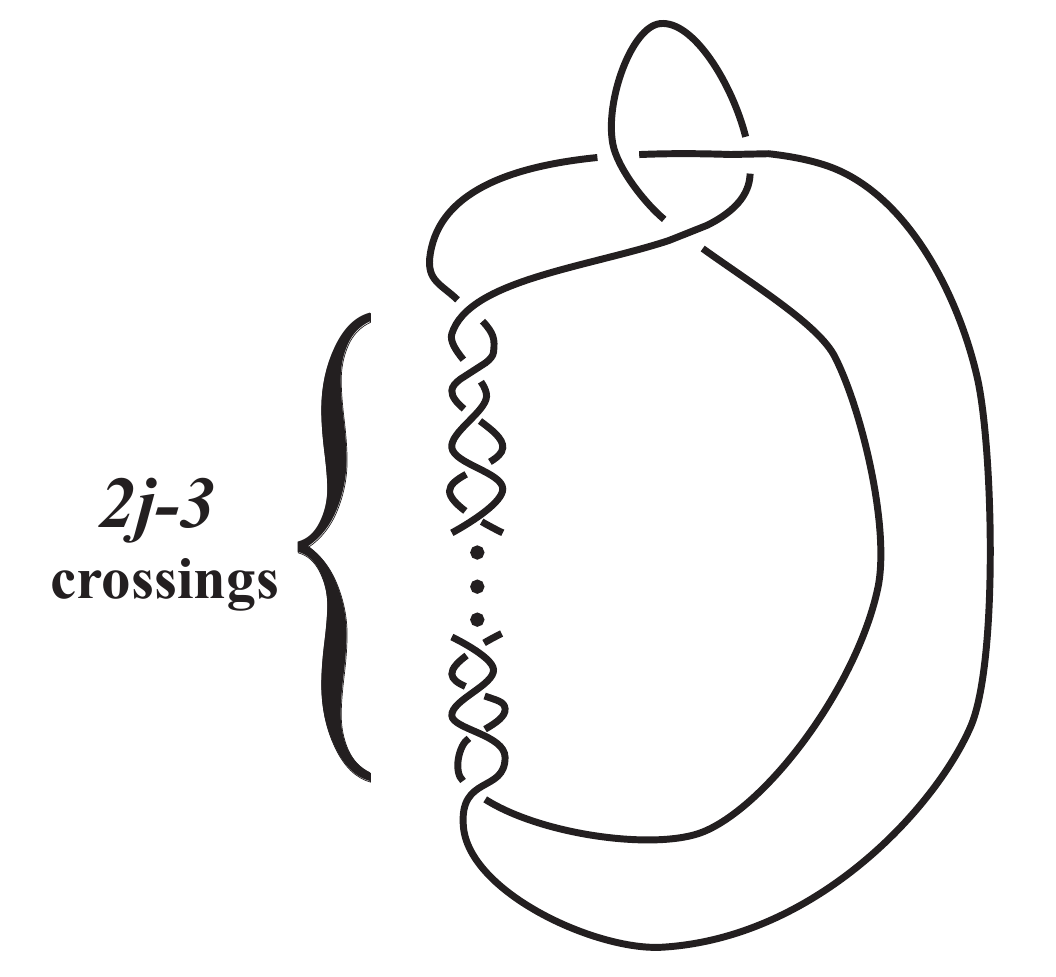} \\
\hspace{.5in}(b)\hspace{1.8in}(c)\hspace{1.8in}(d)
\end{center}

\caption{Subfigure (a) shows a shadow with $2j$ crossings, (b) gives a resolution that produces a diagram of the $(2,2j-1)$ torus knot, (c) shows a resolution to a diagram of $4_1$, and (d) shows a diagram of $5_2$.}\label{kk_pf}
\end{figure}

\section{Conclusion}\label{conclusion}

We conclude with some of our favorite lingering questions and ideas. 

\begin{enumerate}
\item Are there fertile knots with crossing number greater than seven?

It seems highly unlikely that, if there are no fertile knots with crossing number 8, 9, or 10, there will suddenly appear a fertile knot with crossing number greater than 10. We have not conclusively ruled this possibility out, however. How could one prove that no more fertile knots exist?

\item Can the results of Section \ref{families} be generalized? In other words, what more can we say about our favorite families of knots?

We have yet to explore other families of rational knots, pretzel knots, or Montesinos knots. In general, are there similar descendant relationships between families of knots other than twist knots and $(2,p)$ torus knots?

\item For an arbitrary $n$-crossing knot $K$, what proportion of $m$-crossing knots (where $m>n$) should we expect to be parents of $K$? 

For example, each alternating 8-crossing knot {\em fails} to be a descendant of between 33 and 41 of the 9-crossing alternating knots. In other words, each alternating 8-crossing knot has at most eight 9-crossing parents. Of the 123 10-crossing alternating knots, any given alternating 8-crossing knot will fail to be a descendant of between 73 and 116. 

\item What more can we say about barriers to fertility?

For instance, for a given $n$, what is the smallest set,  $\mathcal{S}$, of alternating $n$-crossing knots which has the property that no alternating $(n+1)$-crossing knot has all members of $\mathcal{S}$ as descendants?

\item What more can we say about  $(n,m)$-fertility in general and $(k,k)$-fertility in particular?

We saw that $3_1$, $4_1$, and $5_2$ are $(k,k)$-fertile for infinitely many $k$. Are there other knot types that are $(k,k)$-fertile for all sufficiently large $k$ or all sufficiently large odd or even $k$?

\end{enumerate}

\section*{Acknowledgments}

First, we would like to thank Ken Millett, Erica Flapan, Claus Ernst, Harrison Chapman, and Colin Adams for their suggestions that led to improvements of this paper. We express our deep gratitude to the National Science Foundation for supporting this work through grants DMS \#1460537 and \#1418869. This work was also supported by a grant from the Simons Foundation (\#426566, Allison Henrich) as well as the Clare Boothe Luce Foundation and the Seattle University College of Science and Engineering. Finally, this project was inspired by an observation of Inga Johnson's. We thank her for the inspiration.

\bibliography{revisionv5-final}
\bibliographystyle{plain}


\begin{table}
  \caption{Fertility levels for knot types through 10 crossings}
  \medskip
  
  \centering
\begin{tabular}{c|c||c|c||c|c||c|c||c|c||c|c||c|c}
  $K$ & $F$ & $K$ & $F$ & $K$ & $F$ & $K$ & $F$ & $K$ & $F$ & $K$ & $F$ & $K$ & $F$\\
  \hline
$3_1$ & 3 &          $9_{1}$ & 3 &    $9_{41}$ & 5 &         $10_{26}$ & 7 &   $10_{66}$ & 6 &   $10_{106}$ & 5 &  $10_{146}$ & 7\\     
$4_1$ & 4 &          $9_{2}$ & 4 &    $9_{42}$ & 6 &         $10_{27}$ & 6 &   $10_{67}$ & 6 &   $10_{107}$ & 6 &  $10_{147}$ & 7\\     
$5_1$ & 3 &          $9_{3}$ & 5 &    $9_{43}$ & 6 &         $10_{28}$ & 6 &   $10_{68}$ & 6 &   $10_{108}$ & 6 &  $10_{148}$ & 6\\     
$5_2$ & 4 &          $9_{4}$ & 5 &    $9_{44}$ & 6 &         $10_{29}$ & 6 &   $10_{69}$ & 7 &   $10_{109}$ & 5 &  $10_{149}$ & 6\\     
$6_1$ & 4 &          $9_{5}$ & 4 &    $9_{45}$ & 6 &         $10_{30}$ & 6 &   $10_{70}$ & 6 &   $10_{110}$ & 6 &  $10_{150}$ & 6\\     
$6_2$ & 5 &          $9_{6}$ & 5 &    $9_{46}$ & 6 &         $10_{31}$ & 6 &   $10_{71}$ & 6 &   $10_{111}$ & 6 &  $10_{151}$ & 6\\     
$6_3$ & 5 &          $9_{7}$ & 6 &    $9_{47}$ & 6 &         $10_{32}$ & 6 &   $10_{72}$ & 6 &   $10_{112}$ & 5 &  $10_{152}$ & 5\\     
$3_1\#3_1$ & 3 &     $9_{8}$ & 6 &    $9_{48}$ & 6 &         $10_{33}$ & 6 &   $10_{73}$ & 6 &   $10_{113}$ & 6 &  $10_{153}$ & 6\\     
$7_{1}$ & 3 &        $9_{9}$ & 5 &    $9_{49}$ & 6 &         $10_{34}$ & 6 &   $10_{74}$ & 5 &   $10_{114}$ & 6 &  $10_{154}$ & 6\\     
$7_{2}$ & 4 &        $9_{10}$ & 5 &   $3_1\#6_1$ & 4 &       $10_{35}$ & 6 &   $10_{75}$ & 5 &   $10_{115}$ & 6 &  $10_{155}$ & 5\\     
$7_{3}$ & 5 &        $9_{11}$ & 6 &   $3_1\#6_2$ & 5 &       $10_{36}$ & 6 &   $10_{76}$ & 6 &   $10_{116}$ & 5 &  $10_{156}$ & 6\\     
$7_{4}$ & 4 &        $9_{12}$ & 6 &   $3_1\#6_3$ & 5 &       $10_{37}$ & 6 &   $10_{77}$ & 6 &   $10_{117}$ & 6 &  $10_{157}$ & 5\\     
$7_{5}$ & 5 &        $9_{13}$ & 6 &   $4_1\#5_1$ & 4 &       $10_{38}$ & 6 &   $10_{78}$ & 6 &   $10_{118}$ & 5 &  $10_{158}$ & 6\\     
$7_{6}$ & 6 &        $9_{14}$ & 5 &   $4_1\#5_2$ & 4 &       $10_{39}$ & 7 &   $10_{79}$ & 5 &   $10_{119}$ & 6 &  $10_{159}$ & 5\\     
$7_{7}$ & 5 &        $9_{15}$ & 6 &   $3_1\#3_1\#3_1$ & 3 &  $10_{40}$ & 7 &   $10_{80}$ & 6 &   $10_{120}$ & 6 &  $10_{160}$ & 6\\     
$3_{1}\#4_{1}$ & 4 & $9_{16}$ & 5 &   $10_{1}$ & 4 &         $10_{41}$ & 7 &   $10_{81}$ & 6 &   $10_{121}$ & 6 &  $10_{161}$ & 7\\     
$8_{1}$ & 4 &        $9_{17}$ & 5 &   $10_{2}$ & 5 &         $10_{42}$ & 7 &   $10_{82}$ & 5 &   $10_{122}$ & 6 &  $10_{162}$ & 6\\     
$8_{2}$ & 5 &        $9_{18}$ & 6 &   $10_{3}$ & 4 &         $10_{43}$ & 6 &   $10_{83}$ & 6 &   $10_{123}$ & 5 &  $10_{163}$ & 6\\     
$8_{3}$ & 4 &        $9_{19}$ & 6 &   $10_{4}$ & 5 &         $10_{44}$ & 7 &   $10_{84}$ & 6 &   $10_{124}$ & 5 &  $10_{164}$ & 6\\     
$8_{4}$ & 5 &        $9_{20}$ & 6 &   $10_{5}$ & 5 &         $10_{45}$ & 7 &   $10_{85}$ & 5 &   $10_{125}$ & 6 &  $10_{165}$ & 6\\     
$8_{5}$ & 5 &        $9_{21}$ & 6 &   $10_{6}$ & 6 &         $10_{46}$ & 5 &   $10_{86}$ & 6 &   $10_{126}$ & 6 &  $3_1\#7_1$ & 3\\     
$8_{6}$ & 6 &        $9_{22}$ & 6 &   $10_{7}$ & 5 &         $10_{47}$ & 5 &   $10_{87}$ & 6 &   $10_{127}$ & 6 &  $3_1\#7_2$ & 4\\     
$8_{7}$ & 5 &        $9_{23}$ & 6 &   $10_{8}$ & 5 &         $10_{48}$ & 5 &   $10_{88}$ & 6 &   $10_{128}$ & 6 &  $3_1\#7_3$ & 5\\     
$8_{8}$ & 6 &        $9_{24}$ & 6 &   $10_{9}$ & 5 &         $10_{49}$ & 6 &   $10_{89}$ & 6 &   $10_{129}$ & 6 &  $3_1\#7_4$ & 4\\     
$8_{9}$ & 5 &        $9_{25}$ & 6 &   $10_{10}$ & 6 &        $10_{50}$ & 6 &   $10_{90}$ & 6 &   $10_{130}$ & 6 &  $3_1\#7_5$ & 5\\     
$8_{10}$ & 5 &       $9_{26}$ & 6 &   $10_{11}$ & 6 &        $10_{51}$ & 6 &   $10_{91}$ & 5 &   $10_{131}$ & 6 &  $3_1\#7_6$ & 6\\     
$8_{11}$ & 5 &       $9_{27}$ & 7 &   $10_{12}$ & 6 &        $10_{52}$ & 6 &   $10_{92}$ & 6 &   $10_{132}$ & 6 &  $3_1\#7_7$ & 5\\     
$8_{12}$ & 6 &       $9_{28}$ & 6 &   $10_{13}$ & 6 &        $10_{53}$ & 6 &   $10_{93}$ & 6 &   $10_{133}$ & 6 &  $4_1\#6_1$ & 4\\     
$8_{13}$ & 6 &       $9_{29}$ & 6 &   $10_{14}$ & 6 &        $10_{54}$ & 6 &   $10_{94}$ & 5 &   $10_{134}$ & 6 &  $4_1\#6_2$ & 5\\     
$8_{14}$ & 6 &       $9_{30}$ & 6 &   $10_{15}$ & 6 &        $10_{55}$ & 6 &   $10_{95}$ & 6 &   $10_{135}$ & 6 &  $4_1\#6_3$ & 5\\     
$8_{15}$ & 6 &       $9_{31}$ & 6 &   $10_{16}$ & 5 &        $10_{56}$ & 6 &   $10_{96}$ & 6 &   $10_{136}$ & 6 &  $5_1\#5_1$ & 3\\     
$8_{16}$ & 5 &       $9_{32}$ & 6 &   $10_{17}$ & 5 &        $10_{57}$ & 6 &   $10_{97}$ & 6 &   $10_{137}$ & 6 &  $5_1\#5_2$ & 5\\     
$8_{17}$ & 5 &       $9_{33}$ & 6 &   $10_{18}$ & 6 &        $10_{58}$ & 6 &   $10_{98}$ & 6 &   $10_{138}$ & 6 &  $5_2\#5_2$ & 4\\     
$8_{18}$ & 5 &       $9_{34}$ & 6 &   $10_{19}$ & 6 &        $10_{59}$ & 6 &   $10_{99}$ & 5 &   $10_{139}$ & 5 &  $3_1\#3_1\#4_1$ & 4\\
$8_{19}$ & 5 &       $9_{35}$ & 4 &   $10_{20}$ & 6 &        $10_{60}$ & 6 &   $10_{100}$ & 5 &  $10_{140}$ & 6 &  & \\
$8_{20}$ & 6 &       $9_{36}$ & 6 &   $10_{21}$ & 5 &        $10_{61}$ & 5 &   $10_{101}$ & 6 &  $10_{141}$ & 6 &  & \\
$8_{21}$ & 6 &       $9_{37}$ & 5 &   $10_{22}$ & 6 &        $10_{62}$ & 5 &   $10_{102}$ & 6 &  $10_{142}$ & 6 &  & \\
$3_1\#5_1$ & 3 &     $9_{38}$ & 6 &   $10_{23}$ & 7 &        $10_{63}$ & 6 &   $10_{103}$ & 6 &  $10_{143}$ & 6 &  & \\
$3_1\#5_2$ & 4 &     $9_{39}$ & 6 &   $10_{24}$ & 6 &        $10_{64}$ & 5 &   $10_{104}$ & 5 &  $10_{144}$ & 6 &  & \\
$4_1\#4_1$ & 4 &     $9_{40}$ & 5 &   $10_{25}$ & 7 &        $10_{65}$ & 6 &   $10_{105}$ & 6 &  $10_{145}$ & 6 &  & \\
\end{tabular}
\label{fertilitytable}
\end{table}

\begin{table}
  \caption{Table of siblings through 8- and 9-crossing knot types}
  \bigskip
  
  \centering
\begin{tabular}{c||l}
  knot & siblings\\
  \hline
$8_{5}$ & $8_{19}$ $8_{20}$\\
$8_{10}$ & $8_{19}$ $8_{20}$ $8_{21}$\\
$8_{15}$ & $8_{20}$ $8_{21}$\\
$8_{16}$ & $8_{19}$ $8_{20}$ $8_{21}$\\
$8_{17}$ & $8_{19}$ $8_{20}$ $8_{21}$\\
$8_{18}$ & $8_{19}$ $8_{20}$\\
$8_{19}$ & $8_{5}$ $8_{10}$ $8_{16}$ $8_{17}$ $8_{18}$ $8_{20}$ $8_{21}$\\
$8_{20}$ & $8_{5}$ $8_{10}$ $8_{15}$ $8_{16}$ $8_{17}$ $8_{18}$ $8_{19}$ $8_{21}$\\
$8_{21}$ & $8_{10}$ $8_{15}$ $8_{16}$ $8_{17}$ $8_{19}$ $8_{20}$\\
\hline
$9_{22}$ & $9_{42}$ $9_{43}$ $9_{45}$\\
$9_{25}$ & $9_{42}$ $9_{44}$ $9_{45}$\\
$9_{29}$ & $9_{42}$ $9_{43}$ $9_{44}$ $9_{46}$\\
$9_{30}$ & $9_{43}$ $9_{44}$ $9_{45}$\\
$9_{32}$ & $9_{42}$ $9_{43}$ $9_{44}$ $9_{45}$\\
$9_{33}$ & $9_{42}$ $9_{43}$ $9_{44}$ $9_{45}$\\
$9_{34}$ & $9_{42}$ $9_{43}$ $9_{44}$ $9_{46}$ $9_{47}$\\
$9_{35}$ & $9_{46}$\\
$9_{36}$ & $9_{42}$ $9_{43}$ $9_{44}$\\
$9_{37}$ & $9_{46}$ $9_{48}$\\
$9_{38}$ & $9_{42}$ $9_{44}$ $9_{45}$ $9_{48}$\\
$9_{39}$ & $9_{42}$ $9_{44}$ $9_{46}$ $9_{48}$ $9_{49}$\\
$9_{40}$ & $9_{42}$ $9_{46}$ $9_{47}$\\
$9_{41}$ & $9_{42}$ $9_{46}$ $9_{49}$\\
$9_{42}$ & $9_{22}$ $9_{25}$ $9_{29}$ $9_{32}$ $9_{33}$ $9_{34}$ $9_{36}$ $9_{38}$ $9_{39}$ $9_{40}$ $9_{41}$ $9_{43}$ $9_{44}$ $9_{45}$ $9_{46}$ $9_{47}$ $9_{48}$ $9_{49}$\\
$9_{43}$ & $9_{22}$ $9_{29}$ $9_{30}$ $9_{32}$ $9_{33}$ $9_{34}$ $9_{36}$ $9_{42}$ $9_{44}$ $9_{45}$ $9_{46}$ $9_{47}$\\
$9_{44}$ & $9_{25}$ $9_{29}$ $9_{30}$ $9_{32}$ $9_{33}$ $9_{34}$ $9_{36}$ $9_{38}$ $9_{39}$ $9_{42}$ $9_{43}$ $9_{45}$ $9_{46}$ $9_{47}$ $9_{48}$ $9_{49}$\\
$9_{45}$ & $9_{22}$ $9_{25}$ $9_{30}$ $9_{32}$ $9_{33}$ $9_{38}$ $9_{42}$ $9_{43}$ $9_{44}$ $9_{48}$\\
$9_{46}$ & $9_{29}$ $9_{34}$ $9_{35}$ $9_{37}$ $9_{39}$ $9_{40}$ $9_{41}$ $9_{42}$ $9_{43}$ $9_{44}$ $9_{47}$ $9_{48}$ $9_{49}$\\
$9_{47}$ & $9_{34}$ $9_{40}$ $9_{42}$ $9_{43}$ $9_{44}$ $9_{46}$\\
$9_{48}$ & $9_{37}$ $9_{38}$ $9_{39}$ $9_{42}$ $9_{44}$ $9_{45}$ $9_{46}$ $9_{49}$\\
$9_{49}$ & $9_{39}$ $9_{41}$ $9_{42}$ $9_{44}$ $9_{46}$ $9_{48}$\\
\end{tabular}
\label{siblingstable}
\end{table}

\begin{table}
  \caption{Maximal $m$ for $(m,3)$-fertility in knot types through 3 crossings}
    \bigskip
    
  \centering
  \begin{tabular}{c|c||c|c}
    $K$ & $m$ & $K$ & $m$ \\
    \hline
    $0_{1}$ & 3 & $3_{1}$ & 3 \\
  \end{tabular}
  \label{m3fertilitytable}
\end{table}

\begin{table}
  \caption{Maximal $m$ for $(m,4)$-fertility in knot types through 4 crossings}
    \bigskip
    
  \centering
  \begin{tabular}{c|c||c|c||c|c}
    $K$ & $m$ & $K$ & $m$ & $K$ & $m$ \\
    \hline
    $0_{1}$ & 4 & $3_{1}$ & 4 & $4_{1}$ & 4 \\
  \end{tabular}
  \label{m4fertilitytable}
\end{table}

\begin{table}
  \caption{Maximal $m$ for $(m,5)$-fertility in knot types through 5 crossings}
    \bigskip
    
  \centering
  \begin{tabular}{c|c||c|c||c|c||c|c||c|c}
    $K$ & $m$ & $K$ & $m$ & $K$ & $m$ & $K$ & $m$ & $K$ & $m$ \\
    \hline
    $0_{1}$ & 5 & $3_{1}$ & 5 & $4_{1}$ & 4 & $5_{1}$ & 3 & $5_{2}$ & 4 \\
  \end{tabular}
  \label{m5fertilitytable}
\end{table}

\begin{table}
  \caption{Maximal $m$ for $(m,6)$-fertility in knot types through 6 crossings}
    \bigskip
    
  \centering
  \begin{tabular}{c|c||c|c||c|c||c|c||c|c}
    $K$ & $m$ & $K$ & $m$ & $K$ & $m$ & $K$ & $m$ & $K$ & $m$ \\
    \hline
    $0_{1}$ & 6 & $4_{1}$ & 6 & $5_{2}$ & 6 & $6_{2}$ & 5 & $3_1\#3_1$ & 3 \\
    $3_{1}$ & 6 & $5_{1}$ & 5 & $6_{1}$ & 4 & $6_{3}$ & 5 & & \\
  \end{tabular}
  \label{m6fertilitytable}
\end{table}

\begin{table}
  \caption{Maximal $m$ for $(m,7)$-fertility in knot types through 7 crossings}
    \bigskip
    
  \centering
  \begin{tabular}{c|c||c|c||c|c||c|c||c|c||c|c}
    $K$ & $m$ & $K$ & $m$ & $K$ & $m$ & $K$ & $m$ & $K$ & $m$ & $K$ & $m$ \\
    \hline
    $0_{1}$ & 7 & $5_{1}$ & 6 & $6_{2}$ & 6 & $7_{1}$ & 3 & $7_{4}$ & 4 & $7_{7}$ & 5 \\
    $3_{1}$ & 7 & $5_{2}$ & 6 & $6_{3}$ & 6 & $7_{2}$ & 4 & $7_{5}$ & 5 & $3_1\#4_1$ & 4 \\
    $4_{1}$ & 6 & $6_{1}$ & 6 & $3_1\#3_1$ & 4 & $7_{3}$ & 5 & $7_{6}$ & 6 & & \\
  \end{tabular}
  \label{m7fertilitytable}
\end{table}

\begin{table}
  \caption{Maximal $m$ for $(m,8)$-fertility in knot types through 8 crossings}
    \bigskip
    
  \centering
  \begin{tabular}{c|c||c|c||c|c||c|c||c|c||c|c||c|c}
    $K$ & $m$ & $K$ & $m$ & $K$ & $m$ & $K$ & $m$ & $K$ & $m$ & $K$ & $m$ & $K$ & $m$ \\
    \hline
    $0_{1}$ & 8 & $6_{2}$ & 7 & $7_{4}$ & 6 & $8_{2}$ & 5 & $8_{8}$ & 6 & $8_{14}$ & 6 & $8_{20}$ & 6 \\
    $3_{1}$ & 8 & $6_{3}$ & 7 & $7_{5}$ & 6 & $8_{3}$ & 4 & $8_{9}$ & 5 & $8_{15}$ & 6 & $8_{21}$ & 6 \\
    $4_{1}$ & 8 & $3_1\#3_1$ & 6 & $7_{6}$ & 6 & $8_{4}$ & 5 & $8_{10}$ & 5 & $8_{16}$ & 5 & $3_1\#5_1$ & 3 \\
    $5_{1}$ & 7 & $7_{1}$ & 5 & $7_{7}$ & 6 & $8_{5}$ & 5 & $8_{11}$ & 5 & $8_{17}$ & 5 & $3_1\#5_2$ & 4 \\
    $5_{2}$ & 8 & $7_{2}$ & 6 & $3_1\#4_1$ & 4 & $8_{6}$ & 6 & $8_{12}$ & 6 & $8_{18}$ & 5 & $4_1\#4_1$ & 4 \\
    $6_{1}$ & 6 & $7_{3}$ & 7 & $8_{1}$ & 4 & $8_{7}$ & 5 & $8_{13}$ & 6 & $8_{19}$ & 5 & & \\
  \end{tabular}
  \label{m8fertilitytable}
\end{table}

\begin{table}
  \caption{Maximal $m$ for $(m,9)$-fertility in knot types through 9 crossings}
    \bigskip
    
  \centering
  \begin{tabular}{c|c||c|c||c|c||c|c||c|c||c|c||c|c}
    $K$ & $m$ & $K$ & $m$ & $K$ & $m$ & $K$ & $m$ & $K$ & $m$ & $K$ & $m$ & $K$ & $m$ \\
    \hline
    $0_{1}$ & 9 & $7_{6}$ & 7 & $8_{12}$ & 6 & $9_{2}$ & 4 & $9_{16}$ & 5 & $9_{30}$ & 6 & $9_{44}$ & 6 \\
    $3_{1}$ & 9 & $7_{7}$ & 7 & $8_{13}$ & 7 & $9_{3}$ & 5 & $9_{17}$ & 5 & $9_{31}$ & 6 & $9_{45}$ & 6 \\
    $4_{1}$ & 8 & $3_1\#4_1$ & 6 & $8_{14}$ & 7 & $9_{4}$ & 5 & $9_{18}$ & 6 & $9_{32}$ & 6 & $9_{46}$ & 6 \\
    $5_{1}$ & 8 & $8_{1}$ & 6 & $8_{15}$ & 6 & $9_{5}$ & 4 & $9_{19}$ & 6 & $9_{33}$ & 6 & $9_{47}$ & 6 \\
    $5_{2}$ & 8 & $8_{2}$ & 7 & $8_{16}$ & 6 & $9_{6}$ & 5 & $9_{20}$ & 6 & $9_{34}$ & 6 & $9_{48}$ & 6 \\
    $6_{1}$ & 7 & $8_{3}$ & 6 & $8_{17}$ & 6 & $9_{7}$ & 6 & $9_{21}$ & 6 & $9_{35}$ & 4 & $9_{49}$ & 6 \\
    $6_{2}$ & 7 & $8_{4}$ & 7 & $8_{18}$ & 5 & $9_{8}$ & 6 & $9_{22}$ & 6 & $9_{36}$ & 6 & $3_1\#6_1$ & 4 \\
    $6_{3}$ & 8 & $8_{5}$ & 6 & $8_{19}$ & 6 & $9_{9}$ & 5 & $9_{23}$ & 6 & $9_{37}$ & 5 & $3_1\#6_2$ & 5 \\
    $3_1\#3_1$ & 6 & $8_{6}$ & 7 & $8_{20}$ & 6 & $9_{10}$ & 5 & $9_{24}$ & 6 & $9_{38}$ & 6 & $3_1\#6_3$ & 5 \\
    $7_{1}$ & 7 & $8_{7}$ & 7 & $8_{21}$ & 6 & $9_{11}$ & 6 & $9_{25}$ & 6 & $9_{39}$ & 6 & $4_1\#5_1$ & 4 \\
    $7_{2}$ & 7 & $8_{8}$ & 7 & $3_1\#5_1$ & 5 & $9_{12}$ & 6 & $9_{26}$ & 6 & $9_{40}$ & 5 & $4_1\#5_2$ & 4 \\
    $7_{3}$ & 7 & $8_{9}$ & 6 & $3_1\#5_2$ & 6 & $9_{13}$ & 6 & $9_{27}$ & 7 & $9_{41}$ & 5 & $3_1\#3_1\#3_1$ & 3 \\
    $7_{4}$ & 7 & $8_{10}$ & 6 & $4_1\#4_1$ & 4 & $9_{14}$ & 5 & $9_{28}$ & 6 & $9_{42}$ & 6 & & \\
    $7_{5}$ & 7 & $8_{11}$ & 7 & $9_{1}$ & 3 & $9_{15}$ & 6 & $9_{29}$ & 6 & $9_{43}$ & 6 & & \\
  \end{tabular}
  \label{m9fertilitytable}
\end{table}

\begin{table}
  \caption{Maximal $m$ for $(m,10)$-fertility in knot types through 10 crossings}
    \bigskip
{\small     
  \centering
  \begin{tabular}{c|c||c|c||c|c||c|c||c|c||c|c||c|c}
    $K$ & $m$ & $K$ & $m$ & $K$ & $m$ & $K$ & $m$ & $K$ & $m$ & $K$ & $m$ & $K$ & $m$ \\
    \hline
    $0_{1}$ & 10 & $4_1\#4_1$ & 6 & $9_{40}$ & 5 & $10_{25}$ & 7 & $10_{65}$ & 6 & $10_{105}$ & 6 & $10_{145}$ & 6 \\
    $3_{1}$ & 10 & $9_{1}$ & 5 & $9_{41}$ & 6 & $10_{26}$ & 7 & $10_{66}$ & 6 & $10_{106}$ & 5 & $10_{146}$ & 7 \\
    $4_{1}$ & 10 & $9_{2}$ & 6 & $9_{42}$ & 6 & $10_{27}$ & 6 & $10_{67}$ & 6 & $10_{107}$ & 6 & $10_{147}$ & 7 \\
    $5_{1}$ & 9 & $9_{3}$ & 7 & $9_{43}$ & 6 & $10_{28}$ & 6 & $10_{68}$ & 6 & $10_{108}$ & 6 & $10_{148}$ & 6 \\
    $5_{2}$ & 10 & $9_{4}$ & 7 & $9_{44}$ & 6 & $10_{29}$ & 6 & $10_{69}$ & 7 & $10_{109}$ & 5 & $10_{149}$ & 6 \\
    $6_{1}$ & 8 & $9_{5}$ & 6 & $9_{45}$ & 6 & $10_{30}$ & 6 & $10_{70}$ & 6 & $10_{110}$ & 6 & $10_{150}$ & 6 \\
    $6_{2}$ & 9 & $9_{6}$ & 7 & $9_{46}$ & 7 & $10_{31}$ & 6 & $10_{71}$ & 6 & $10_{111}$ & 6 & $10_{151}$ & 6 \\
    $6_{3}$ & 8 & $9_{7}$ & 7 & $9_{47}$ & 6 & $10_{32}$ & 6 & $10_{72}$ & 6 & $10_{112}$ & 5 & $10_{152}$ & 5 \\
    $3_1\#3_1$ & 8 & $9_{8}$ & 7 & $9_{48}$ & 7 & $10_{33}$ & 6 & $10_{73}$ & 6 & $10_{113}$ & 6 & $10_{153}$ & 6 \\
    $7_{1}$ & 8 & $9_{9}$ & 7 & $9_{49}$ & 6 & $10_{34}$ & 6 & $10_{74}$ & 5 & $10_{114}$ & 6 & $10_{154}$ & 6 \\
    $7_{2}$ & 8 & $9_{10}$ & 7 & $3_1\#6_1$ & 6 & $10_{35}$ & 6 & $10_{75}$ & 5 & $10_{115}$ & 6 & $10_{155}$ & 5 \\
    $7_{3}$ & 8 & $9_{11}$ & 7 & $3_1\#6_2$ & 6 & $10_{36}$ & 6 & $10_{76}$ & 6 & $10_{116}$ & 5 & $10_{156}$ & 6 \\
    $7_{4}$ & 8 & $9_{12}$ & 7 & $3_1\#6_3$ & 6 & $10_{37}$ & 6 & $10_{77}$ & 6 & $10_{117}$ & 6 & $10_{157}$ & 5 \\
    $7_{5}$ & 8 & $9_{13}$ & 7 & $4_1\#5_1$ & 5 & $10_{38}$ & 6 & $10_{78}$ & 6 & $10_{118}$ & 5 & $10_{158}$ & 6 \\
    $7_{6}$ & 8 & $9_{14}$ & 7 & $4_1\#5_2$ & 6 & $10_{39}$ & 7 & $10_{79}$ & 5 & $10_{119}$ & 6 & $10_{159}$ & 5 \\
    $7_{7}$ & 8 & $9_{15}$ & 7 & $3_1\#3_1\#3_1$ & 4 & $10_{40}$ & 7 & $10_{80}$ & 6 & $10_{120}$ & 6 & $10_{160}$ & 6 \\
    $3_1\#4_1$ & 6 & $9_{16}$ & 6 & $10_{1}$ & 4 & $10_{41}$ & 7 & $10_{81}$ & 6 & $10_{121}$ & 6 & $10_{161}$ & 7 \\
    $8_{1}$ & 7 & $9_{17}$ & 7 & $10_{2}$ & 5 & $10_{42}$ & 7 & $10_{82}$ & 5 & $10_{122}$ & 6 & $10_{162}$ & 6 \\
    $8_{2}$ & 7 & $9_{18}$ & 7 & $10_{3}$ & 4 & $10_{43}$ & 6 & $10_{83}$ & 6 & $10_{123}$ & 5 & $10_{163}$ & 6 \\
    $8_{3}$ & 7 & $9_{19}$ & 7 & $10_{4}$ & 5 & $10_{44}$ & 7 & $10_{84}$ & 6 & $10_{124}$ & 5 & $10_{164}$ & 6 \\
    $8_{4}$ & 7 & $9_{20}$ & 7 & $10_{5}$ & 5 & $10_{45}$ & 7 & $10_{85}$ & 5 & $10_{125}$ & 6 & $10_{165}$ & 6 \\
    $8_{5}$ & 7 & $9_{21}$ & 7 & $10_{6}$ & 6 & $10_{46}$ & 5 & $10_{86}$ & 6 & $10_{126}$ & 6 & $3_1\#7_1$ & 3 \\
    $8_{6}$ & 7 & $9_{22}$ & 6 & $10_{7}$ & 5 & $10_{47}$ & 5 & $10_{87}$ & 6 & $10_{127}$ & 6 & $3_1\#7_2$ & 4 \\
    $8_{7}$ & 8 & $9_{23}$ & 7 & $10_{8}$ & 5 & $10_{48}$ & 5 & $10_{88}$ & 6 & $10_{128}$ & 6 & $3_1\#7_3$ & 5 \\
    $8_{8}$ & 8 & $9_{24}$ & 7 & $10_{9}$ & 5 & $10_{49}$ & 6 & $10_{89}$ & 6 & $10_{129}$ & 6 & $3_1\#7_4$ & 4 \\
    $8_{9}$ & 8 & $9_{25}$ & 6 & $10_{10}$ & 6 & $10_{50}$ & 6 & $10_{90}$ & 6 & $10_{130}$ & 6 & $3_1\#7_5$ & 5 \\
    $8_{10}$ & 7 & $9_{26}$ & 7 & $10_{11}$ & 6 & $10_{51}$ & 6 & $10_{91}$ & 5 & $10_{131}$ & 6 & $3_1\#7_6$ & 6 \\
    $8_{11}$ & 7 & $9_{27}$ & 7 & $10_{12}$ & 6 & $10_{52}$ & 6 & $10_{92}$ & 6 & $10_{132}$ & 6 & $3_1\#7_7$ & 5 \\
    $8_{12}$ & 8 & $9_{28}$ & 6 & $10_{13}$ & 6 & $10_{53}$ & 6 & $10_{93}$ & 6 & $10_{133}$ & 6 & $4_1\#6_1$ & 4 \\
    $8_{13}$ & 8 & $9_{29}$ & 6 & $10_{14}$ & 6 & $10_{54}$ & 6 & $10_{94}$ & 5 & $10_{134}$ & 6 & $4_1\#6_2$ & 5 \\
    $8_{14}$ & 8 & $9_{30}$ & 6 & $10_{15}$ & 6 & $10_{55}$ & 6 & $10_{95}$ & 6 & $10_{135}$ & 6 & $4_1\#6_3$ & 5 \\
    $8_{15}$ & 8 & $9_{31}$ & 7 & $10_{16}$ & 5 & $10_{56}$ & 6 & $10_{96}$ & 6 & $10_{136}$ & 6 & $5_1\#5_1$ & 3 \\
    $8_{16}$ & 7 & $9_{32}$ & 6 & $10_{17}$ & 5 & $10_{57}$ & 6 & $10_{97}$ & 6 & $10_{137}$ & 6 & $5_1\#5_2$ & 5 \\
    $8_{17}$ & 7 & $9_{33}$ & 6 & $10_{18}$ & 6 & $10_{58}$ & 6 & $10_{98}$ & 6 & $10_{138}$ & 6 & $5_2\#5_2$ & 4 \\
    $8_{18}$ & 7 & $9_{34}$ & 6 & $10_{19}$ & 6 & $10_{59}$ & 6 & $10_{99}$ & 5 & $10_{139}$ & 5 & $3_1\#3_1\#4_1$ & 4 \\
    $8_{19}$ & 7 & $9_{35}$ & 6 & $10_{20}$ & 6 & $10_{60}$ & 6 & $10_{100}$ & 5 & $10_{140}$ & 6 & & \\
    $8_{20}$ & 8 & $9_{36}$ & 6 & $10_{21}$ & 5 & $10_{61}$ & 5 & $10_{101}$ & 6 & $10_{141}$ & 6 & & \\
    $8_{21}$ & 8 & $9_{37}$ & 7 & $10_{22}$ & 6 & $10_{62}$ & 5 & $10_{102}$ & 6 & $10_{142}$ & 6 & & \\
    $3_1\#5_1$ & 6 & $9_{38}$ & 6 & $10_{23}$ & 7 & $10_{63}$ & 6 & $10_{103}$ & 6 & $10_{143}$ & 6 & & \\
    $3_1\#5_2$ & 6 & $9_{39}$ & 6 & $10_{24}$ & 6 & $10_{64}$ & 5 & $10_{104}$ & 5 & $10_{144}$ & 6 & & \\
  \end{tabular}
  \label{m10fertilitytable}
  }
\end{table}

\end{document}